\documentclass[11pt,a4paper]{article}

\usepackage{fullpage}
\usepackage{amssymb,amsmath,stmaryrd,amsthm}
\usepackage{accents}

\usepackage{enumerate}
\usepackage[mathscr]{eucal}
\usepackage{epsfig}
\usepackage{pstricks,pst-node,pst-text,pst-3d}
\usepackage{pifont}

\usepackage{textcomp}   % for Euro symbol
\usepackage{multicol}
\usepackage{algorithm}
\usepackage[noend]{algpseudocode}

\theoremstyle{definition}
\newtheorem{theorem}{Theorem}
\newtheorem{lemma}{Lemma}
\newtheorem{remark}{Remark}

\newtheorem{example}{Example}
\newtheorem{definition}{Definition}
\newtheorem{proposition}{Proposition}

\newtheorem{corollary}{Corollary}

\def\1{\mathbf{1}}
\def\cA {\mathscr{A}}
\def\cB {\mathscr{B}}
\def\cC{\mathscr{C}}
\def\cD{\mathscr{D}}
\def\cE{\mathscr{E}}
\def\cF{\mathscr{F}}
\def\cG{\mathscr{G}}

\def\cP{\mathscr{P}}

\def\cS{\mathscr{S}}

\def\cBG {\mathscr{BG}}
\def\RR{{\mathbb R}}

\def\gB{\mathfrak{B}}

\def\Lin{\mathsf{Lin}}
\def \0{{\mathbf{0}}}

\def \conv {\mathrm{conv}}
\def \mymid {\,:\,}
\begin{document}
\title{On the set of balanced games}
\author{Pedro GARCIA-SEGADOR${}^{a}$,  Michel
    GRABISCH${}^{b,c}$\thanks{Corresponding author. Centre d'Economie de
    la Sorbonne, 106-112 Bd de
    l'H\^opital, 75013 Paris, France. Email \tt michel.grabisch@univ-paris1.fr} and Pedro MIRANDA${}^{d}$\\
\normalsize   ${}^a$ National Statistical Institute, Madrid, Spain \\
\normalsize  ${}^b$ Universit\'e Paris I Panth\'eon-Sorbonne, Centre d'Economie
de la Sorbonne\\
\normalsize ${}^c$Paris School of Economics, Paris, France\\
\normalsize  ${}^d$ Complutense University of Madrid, Spain}

\date{}
\maketitle

\begin{abstract}
  We study the geometric structure of the set of cooperative TU-games having a
  nonempty core, characterized by Bondareva and Shapley as balanced games. We
  show that this set is a non-pointed polyhedral cone, and we find the set of
  its extremal rays and facets. This study is also done for the set of balanced
  games whose value for the grand coalition is fixed, which yields an affine
  non-pointed polyhedral cone. Finally, the case of non-negative balanced games
  with fixed value for the grand coalition is tackled. This set is a convex
  polytope, with remarkable properties. We characterize its vertices and facets,
  study the adjacency structure of vertices, develop an algorithm for generating
  vertices in a random uniform way, and show that this polytope is
  combinatorial and its adjacency graph is Hamiltonian. Lastly, we give a
  characterization of the set of games having a core reduced to a singleton.
\end{abstract}


  {\bf MSC2000 subject classification:} Primary 91A12, Secondary 52B05

  {\bf OR/MS subject classification:} Games/group decisions:Cooperative, Mathematics:Sets:Polyhedra

\section{Introduction}
Given a set of players $N$, a game with transferable utility, abbreviated
hereafter as TU-game or simply game, is a mapping $v$ assigning to each subset
$S\subseteq N$ (called a coalition) a quantity $v(S)$, representing, e.g., the
benefit of cooperation of the players in $S$.  The core of a game appears to be
a fundamental concept, introduced by Gillies \cite{gil53}: It is the set of
payment vectors $x\in\RR^N$ to players, such that every coalition $S$ receives
at least the quantity $v(S)$, under the constraint that the total payment which
is distributed is equal to $v(N)$. This set, however, may be empty. Bondareva
\cite{bon63}, and independently Shapley \cite{sha67}, found a sufficient and
necessary condition for a game to have a nonempty core. Games satisfying this
condition are called {\it balanced} by Shapley, because it is based on
particular arrangements of players, called {\it (minimal) balanced collections},
which generalize the notion of partition of a set.  As the number of minimal balanced collections is finite, this condition amounts to check a finite number of linear inequalities, one for each possible minimal balanced collection.
%Interestingly, this condition is finite and consists in checking a number of linear inequalities, one for each possible minimal balanced collection.
This shows that the set of
balanced games is a (closed convex) polyhedron.  To the best of our
  knowledge, there is only one article studying related polyhedra \cite{krst19}, where the facets of the polyhedra of totally balanced, balanced and exact games are studied, but in that paper it is assumed that $v(\varnothing)$ might attain a non-null value, and hence the structures of the polyhedra are different.

The main aim of our paper is precisely to fill this gap. The game-theoretic
motivation behind this mathematical study, apart its own interest, is to solve
(in a future study) the problem of finding the closest balanced game to a
non-balanced game, as that would provide a new solution concept for games with
an empty core. This problem amounts to an orthogonal projection on a polyhedron,
a problem which has no analytical solution in general (see, e.g., \cite{rut17}). The
only hope to solve it is a deep understanding of the structure of the polyhedron
under consideration,  which is exactly what the present paper does.

\medskip

Before elaborating on our findings, we emphasize that this topic is not limited
to (cooperative) game theory. Games as defined above are merely set functions
vanishing on the empty set, and are encountered in many domains linked to
Operations Research, in particular decision theory, voting theory, combinatorial
optimization and reliability (see \cite{gra16} for details). In decision theory,
{\it capacities} \cite{cho53} are a particular class of games which are
monotonic, and represent uncertainty on the set of states of nature (see, e.g.,
Schmeidler \cite{sch89}). In this context, the core is the set of probability
measures which are ``compatible with'' (i.e., eventwise dominating) a given
capacity. Voting theory considers only 0-1-valued monotone games, which are
called {\it simple games} (see, e.g., Peters \cite{pet08}). They represent the
power of coalitions of players in order to win an election. They are in fact
special Boolean functions, while games are nothing other than pseudo-Boolean
functions, as introduced by Hammer (see, e.g., \cite{foha05}).  In combinatorial
optimization, submodular games are considered, as they are encountered for
example as rank function of a polymatroid \cite{edm70}, and the core corresponds to
the base polyhedron of a matroid (see the monograph of Fujishige devoted to this
topic \cite{fuj05b}). Simple games are also used in reliability theory where
they indicate the state of a system (functioning or not), depending on the
states of its components. We mention finally that our study is also linked to
combinatorics and the study of polytopes and polyhedra. Indeed, the concept of minimal balanced
collection generalizes the concept of partitions, and their enumeration remains
an open problem. Also, they are closely related to many geometrical properties
of the games and the core (see \cite{lagrsu23} for details).
As a conclusion, our results are not limited to
(cooperative) game theory but concern many fields of Operations Research, as
well as combinatorics and polyhedra.

\medskip

We split our study in three parts, considering three different sets of balanced
games. The first one is simply the set of all possible balanced games on $N$,
which we denote by $\cBG(n)$, with $|N|=n$. For the second one we impose the
restriction that $v(N)=\alpha$ for some fixed $\alpha\in\RR$. We denote this set
by $\cBG_\alpha(n)$. Lastly, we impose in addition that $v$ should be a
nonnegative function, with $\alpha=1$, w.l.o.g. We denote the set of such balanced game as
$\cBG_+(n)$.  It follows that $\cBG_+(n)$ is a polytope, and we give the complete
characterization of its vertices, as well as an algorithm to randomly generate
them in a uniform way. Interestingly, the number of vertices appears to be a known
integer sequence related to Boolean functions. We also characterize the
adjacency of vertices, and finally we show that this polytope is combinatorial
(in the sense of Naddef and Pulleyblank \cite{napu81}), i.e., the adjacency graph of its
vertices is Hamilton-connected.

We show that $\cBG(n)$ is a non-pointed cone, and identify its lineality space,
all its extremal rays, and all its facets, obtaining thus a complete description
of the polyhedron. The study of $\cBG_\alpha(n)$ is similar, and shows that it
is an affine non-pointed cone. Again, its structure is completely elucidated.

Lastly, we adress the following question: For which balanced games is the core
reduced to a singleton? The question is of interest since a core reduced to a
singleton provides a very simple and handy solution to a game.
We solve completely this question for $\cBG(n)$ and for
$\cBG_\alpha(n)$. Some open questions remain for $\cBG_+(n)$.

The paper is organized as follows. Section~\ref{sec:bc} introduces the necessary
material on games and balanced collections.  In Section~\ref{sec:bg} we study
the cone $\cBG(n)$, its lineality space, its extremal rays and facets, while
Section~\ref{sec:bgalpha} is devoted to the study of the affine cone
$\cBG_\alpha(n)$. Section~\ref{sec:bg+} is devoted to the study of the polytope
$\cBG_+(n)$: dimension, vertices and their enumeration, adjacency, and facets. We also provide a random procedure to generate vertices of this polytope. In Section~\ref{sec:pc} we address the problem of
finding balanced games whose core is reduced to a
singleton. Section 7 is devoted to point some possible applications of these results. Section~\ref{sec:con} gives some concluding remarks.

\section{Basic concepts}\label{sec:bc}
We refer the reader to the monographs \cite{pesu03} and \cite{gra16} for full
details, proofs and references. We limit here ourselves to the essential.
Throughout the paper we consider a (fixed) set $N$ of $n$ players, simply
denoted by $N= \{ 1, \ldots , n\} $. {\it Coalitions} are nonempty subsets of
$N$, denoted by capital letters $S,T$, etc. A {\it TU-game} $(N, v)$ (or simply
a game $v$) is a function $ v: 2^N \rightarrow \mathbb{R}$ satisfying
$v(\varnothing )=0.$ The value $v(S)$ represents the maximal value (benefit)
that the coalition $S$ can guarantee, no matter what players outside $S$ might
do. We will denote by $\cG(N)$ (or simply $\cG(n)$, as no subgame will be
considered) the set of games $v$ on $N.$

For further use, we introduce $\cG_+(n)$, the set of games $v\in\cG(n)$ such
that $v\geqslant 0$ (i.e., $v(S)\geqslant 0$ for all $S\in 2^N$) and $v(N)=1$,
and for every $\alpha\in \RR$ we introduce $\cG_\alpha(n)$, the set of games
$v\in \cG(n)$ such that $v(N)=\alpha$. In addition, we will often use the
following families of games, which are bases of the $(2^n-1)$-dimensional vector
space $\cG(n)$:
\begin{enumerate}
\item The {\it Dirac games} $\delta_S$, $\varnothing\neq S\subseteq N$, defined by
  \[
  \delta_S(T) = \begin{cases}
    1, & \text{ if } T=S\\
    0, & \text{ otherwise.}
    \end{cases}
    \]
  \item The {\it unanimity games} $u_S$,  $\varnothing\neq S\subseteq N$,
    defined by
  \[
  u_S(T) = \begin{cases}
    1, & \text{ if } T\supseteq S\\
    0, & \text{ otherwise.}
    \end{cases}
    \]
\end{enumerate}

Assuming that all players agree to form the grand coalition $N$, we look for a
way to share the benefit $v(N)$ among all players, i.e., for an {\it
  allocation} or {\it payment vector} $x\in \RR^N$, where coordinate $x_i$ indicates the
payoff given to player $i.$ For any coalition $S$, we denote by
$x(S):=\sum_{i\in S} x_i$ the total payoff given to the players in $S$. An
allocation is {\it efficient} if $x(N)=v(N)$. A systematic way of assigning a
set of allocations to a game is called a {\it solution concept}.

In this paper we focus on one of the best known solution concepts, which is the
{\it core} \cite{gil53}. The core is the set of efficient allocations satisfying {\it
  coalitional rationality}, which means that $x(S)\geqslant v(S)$ for all
coalitions $S$. Under this condition, no coalition $S$ has an incentive to leave the
grand coalition $N$ to form a subgame on $S$.  The core of a game
$(N,v)$ is denoted by $C(N,v)$ (or $C(v)$ for short if $N$ is fixed):
\[
C(v):=\{ x\in \mathbb{R}^n: x(S)\geqslant v(S), \forall S\in2^N,x(N)=v(N)\}.
\]
The core is a closed convex polytope that may be empty, as illustrated in the
following example.
\begin{example}\label{ex:1}
The unanimity game $u_S$ has a nonempty core for any $\varnothing\neq S\subseteq
N$. Indeed, $x=\frac{1}{|S|}\1^S$ is a core allocation, where $\1^S$ is the
characteristic vector of $S$, i.e., $\1_i^S=1$ if $i\in S$ and 0 otherwise. Moreover, it can be easily checked that

$$ C(u_S)=\{ x\in\RR^n: x(S)=1, x_j=0, j\not\in S\} .$$

However, any Dirac game $\delta_S$ has an empty core, except when $S=N$ , as it is easy to check.
\end{example}
A condition for nonemptiness of the core has been given by Bondareva
\cite{bon63} and Shapley \cite{sha67}, which we detail below.

A {\it balanced collection}  $\cB$ on $N$ is a family of nonempty subsets of $N$
such that there exist positive (balancing) weights $\lambda^{\cB}_S$, $S\in\cB$, satisfying
\[
 \sum_{\substack{S\in\cB\\S\ni i}} \lambda^{\cB}_S = 1, \forall i\in N.
\]

This notion is an extension of the notion of partition, as any partition is a
balanced collection with balancing weights all equal to 1. A balanced collection
is {\it minimal} if it contains no balanced proper subcollection. We denote by
$\gB^*(n)$ the set of all minimal balanced collections  on a set $N$ of cardinality
$n$, excluding the collection $\{ N\}$. It can be shown that minimal balanced collections
(abbreviated hereafter by m.b.c.) have a
unique set of balancing weights and that their cardinality is at most $n$. Moreover, if $\cB$ is
a m.b.c. with balancing weights $(\lambda^\cB_S)_{S\in \cB}$, then
$\overline{\cB}:=\{N\setminus S\mymid S\in\cB\}$ is also a m.b.c. with
balancing weights
\[
\lambda^{\overline{\cB}}_S=\frac{\lambda^\cB_{N\setminus
    S}}{\sum_{T\in\cB}\lambda^\cB_T-1}.
\]

A game $v$ is {\it balanced} if

\begin{equation}\label{bal}
 \sum_{S\in \cB} \lambda_S^{\cB} v(S) \leqslant v(N), \quad \forall \cB \in \gB^*(n).
\end{equation}

The following result holds \cite{bon63,sha67}.

\begin{theorem}\label{th:bal}
Consider a game $(N, v)$. Then, $C(v)\ne \varnothing $ if and only if $v$ is balanced.
\end{theorem}

We will denote by $\cBG(N)$ the set of games in $\cG(N)$ being
balanced, i.e., the set of games having a nonempty core. For the same reasons as
before, we simplify the notation to $\cBG(n)$.  Hence, applying Theorem~\ref{th:bal} and (\ref{bal}), we get
\begin{equation}\label{eq:bg}
\cBG(n)=\Big\{ v\in \cG(n) \mymid \sum_{S \in \cB} \lambda_S^{\cB} v(S) -v(N) \leqslant 0, \quad \forall  \cB \in \gB^*(n) \Big\} .
\end{equation}
For application purposes and further studies, the
following subsets of $\cBG(n)$ are of interest:
\begin{enumerate}
\item The set of balanced games $v\in\cG_\alpha(n)$ for some $\alpha\in\RR$. We denote by $\cBG_{\alpha }(n)$ the set of such games, i.e.,
\begin{equation}\label{eq:bgalpha}
\cBG_\alpha (n)=\Big\{ v\in \cG_\alpha(n) \mymid \sum_{S \in \cB} \lambda_S^{\cB} v(S) \leqslant \alpha, \quad \forall  \cB \in \gB^*(n) \Big\} .
\end{equation}
The study of this set is motivated as follows: when searching for the closest
balanced game $v'$ to a given non-balanced game $v$, it is natural to impose
that $v'(N)=v(N)$, as this is the total benefit which has to be distributed
among players.
\item The set of balanced games $v\in\cG_+(n)$. We denote by $\cBG_+(n)$ the set
of such games, i.e.,
\begin{equation}\label{eq:bg+}
\cBG_+(n)=\Big\{ v\in \cG_+(n) \mymid \sum_{S \in \cB} \lambda_S^{\cB} v(S) \leqslant 1, \quad \forall \cB \in \gB^*(n) \Big\}.
\end{equation}
Studying this set is motivated by the fact that many applications concern games
with nonnegative values. Therefore it would make no sense to find a closest
balanced game taking negative values.

Observe that if $v\geqslant 0$ and $v(N)=\alpha\neq 0$, then $v$ is balanced iff
$\frac{1}{\alpha}v\in\cBG_+(n)$. This shows that studying the balancedness of
nonnegative games is equivalent to study the balancedness of nonnegative
  games with $v(N)=1$, hence the set $\cBG_+(n)$, except if $v(N)=0$, but the
latter case is trivial as only $v=0$ is balanced.
\end{enumerate}
Note that $\cBG_+(n), \cBG(n)$ and $\cBG_{\alpha }(n)$ are
convex polyhedra. The next sections are devoted to the study of the structure of these polyhedra.

\section{The polyhedron $\cBG(n)$}\label{sec:bg}

Since for any game $v\in \cG(n)$, $v(\varnothing)$ is fixed, we consider $v$
as being a vector in $\RR^{2^N\setminus\{\varnothing\}}$.
%We recall that
%\[
%\cBG(n)=\Big\{v\in\cG(n)\mymid \sum_{S\in\cB}\lambda_S^\cB v(S)-v(N)\leqslant 0, \forall \cB\in\gB^*(n)\Big\} .
%\]
We start by showing a technical result, which will be useful in the sequel.

\begin{lemma}\label{obs:1}
  Let $v$ be such that
  \[
  \sum_{S\in\Pi}v(S)=v(N)
  \]
  for all partitions $\Pi$ of $N$ of the following form: either
  $\Pi=N^\bot:=\{\{1\},\ldots,\{n\}\}$ or $\Pi=\{S,(N\setminus S)^\bot\}$ for
  any $S\subset N$, $|S|>1$, with $T^\bot=\{\{i\},i\in T\}$. Then
  \[
\sum_{S\in\cB}\lambda^\cB_Sv(S)=v(N),\quad \forall \cB\in \gB^*(n).
  \]
\end{lemma}

\begin{proof}
Assume $v$ satisfies the assumption. Take any m.b.c. $\cB$ with balancing
weights $(\lambda^\cB_S)_{S\in\cB}$ and denote
$\overline{\cB}=\{N\setminus S, S\in\cB\}$. Then, $\overline{\cB}$ is a m.b.c. with
weights $\lambda_S^{\overline{\cB}}=\frac{\lambda^\cB_{N\setminus
    S}}{\sum_{T\in\cB}\lambda_T^\cB-1}$ (see Section~\ref{sec:bc}). It follows that

\begin{align*}
\sum_{S\in\cB}\lambda^\cB_Sv(S) & = \sum_{S\in\cB}\lambda_S^\cB\Big(v(N)-\sum_{i\in
  N\setminus S}v(\{i\})\Big) = v(N)\sum_{S\in\cB}\lambda^\cB_S - \sum_{i\in
  N}\left[ v(\{i\})\sum_{\substack{N\setminus S\in\overline{\cB}\\N\setminus S\ni
    i}}\lambda_S^\cB\right] \\
 & = v(N)\sum_{S\in\cB}\lambda^\cB_S - \sum_{i\in
  N}v(\{i\})\sum_{\substack{N\setminus S\in\overline{\cB}\\N\setminus S\ni
    i}} \left[ \lambda^{\overline{\cB}}_{N\setminus
  S}\Big(\sum_{T\in\cB}\lambda_T^\cB-1\Big)\right] \\
 & = v(N)\sum_{S\in\cB}\lambda^\cB_S -
\Big(\sum_{T\in\cB}\lambda_T^\cB-1\Big)\underbrace{\sum_{i\in N}v(\{i\})}_{v(N)} = v(N).
\end{align*}
\end{proof}

Remark that the reciprocal holds as all partitions are minimal balanced collections with balanced weights equal to 1.

Let us first establish the structure of $\cBG(n).$

\subsection{Structure of $\cBG(n)$}

\begin{theorem}\label{th:3}
Let $n\geqslant 2$. Then, $\cBG(n)$ is a $(2^n-1)$-dimensional polyhedral cone,
which is not pointed. Its lineality space has dimension $n$, with basis
$(w_i)_{i\in N}$, with $w_i=\sum_{S\ni i}\delta_S$, where $\delta_S$ is the
Dirac game centered on $S$ (see Section~\ref{sec:bc}).
\end{theorem}

\begin{proof}
We make the proof in two steps.

\begin{enumerate}
\item First, observe that the 0 game is balanced, and if $v$ is balanced, then
  $\alpha v$ with $\alpha\geqslant 0$ is also balanced. Then $\cBG(n)$ is a
  cone. Moreover, it is a polyhedral cone since it is defined by a finite number of linear
  inequalities. Next, observe that no equality can be implied by the system
  (\ref{eq:bg}), as for any $S\ne \emptyset ,$ all coefficients of $v(S)$ are of the same sign. Since in addition there is no equality, the cone is full dimensional.

\item Denoting by $Av\leqslant 0$ the system of inequalities in (\ref{eq:bg}),
  the lineality space is the set of solutions of $Av=0$. From Lemma~\ref{obs:1},
  we infer that this is equivalent to solve the system
  \[
 \sum_{S\in\Pi}v(S)-v(N)=0
  \]
  for all partitions $\Pi$ of $N$ of the following form: either
  $\Pi=N^\bot:=\{\{1\},\ldots,\{n\}\}$ or
  $\Pi=\{S,(N\setminus S)^\bot\}$ for $S\subset N$, $|S|>1$.

  Therefore, we obtain that for any $S\subset N$, $|S|>1$, $v(S) =
  v(N)-\sum_{i\in N\setminus S}v(\{i\})$, and $v(N)=\sum_{i\in N}v(\{i\})$. It
  follows that its set of solutions is, expressing all coordinates in terms of
  those of $v(\{i\})$, $i=1,\ldots,n$,
\[
\{(v_1,\ldots, v_{n},v_1+v_2, \ldots, \sum_{i\in
  S}v_i,\ldots,\sum_{i\in N}v_i)\mymid v_1,\ldots,v_n\in\RR\}.
\]
A basis for this subspace is $\{w_1,\ldots,w_{n}\}$ with
\[
w_i=\sum_{S\ni i}\delta_S.
\]
\end{enumerate}
\end{proof}

The reader may compare this result with Lemma 3.1 in \cite{krst19}, where
  it is proved that modular games form the lineality space of the cone of
  balanced games, totally balanced games, and exact games.

As $\cBG(n)$ is not pointed, it can be decomposed as the direct sum of its lineality
space of dimension $n$ (that we will denote by $\Lin(\cBG(n))$) and a supplementary subspace (not necessarily orthogonal)
of dimension $2^n-n-1$,
which is a pointed cone and whose extremal rays can be found in the usual
way. However, since there is no unique way to choose a supplementary space,
there is no unique representation of $\cBG(n)$ by extremal rays. It seems that in the case of $\cBG(n)$ the orthogonal supplement does not
yield simple results, and we will use instead the supplement where the coordinates
corresponding to singletons are zero. We denote this set as $\cBG^0(n)$ and hence we have

$$\cBG(n)=\Lin(\cBG(n))\oplus \cBG^0(n).$$

We study in the next section the extremal rays of $\cBG(n).$

\subsection{Extremal rays of $\cBG(n)$}

\begin{theorem}\label{th:4}
Let $n\geqslant 2$. The extremal rays of $\cBG(n)$ are
  \begin{itemize}
  \item The $2n$ extremal rays corresponding to $\Lin(\cBG(n))$:
    $w_1,\ldots,w_n,-w_1,\ldots,-w_n$;
  \item $2^n-n-2$ extremal rays of the form $r_S=-\delta_S$, $S\subset N$, $|S|>1$;
  \item $n$ extremal rays of the form
    \[
r_i = \sum_{S\ni i, |S|>1}\delta_S, \quad i\in N.
    \]
  \end{itemize}

This yields in total $2^n+2n-2$ extremal rays.
\end{theorem}

\begin{proof}
We study each part separately.

\begin{enumerate}
\item The $2n$ extremal rays corresponding to $\Lin(\cBG(n))$ space come from the
vectors obtained in Th.~\ref{th:3}.

\item Let us consider $S\subset N$, $|S|>1$ and show that $r_S=-\delta_S$ is an extremal
ray. Let us denote by $r_S(T)$ the coordinate of $r_S$ corresponding to subset
$T$, and similarly for all rays. First, note that $r_S$ is a ray of $\cBG^0(n)$ because it satisfies
the system (\ref{eq:bg}) of inequalities and $r_S(\{i\})=0$ for all $i\in
N$.

Suppose that it is not extremal. Then, there exist two rays $r,r'\in\cBG^0(n)$ not proportional to $r_S$
such that $r_S=r+r'$. Suppose that $r(T)>0$ (w.l.o.g. $r(T)=1$) for some
$T\neq S$, $1<|T|<n$. Then, $r'(T)=-1$. Using the partition $\{T,(N\setminus T)^\bot\}$,
since $r(\{i\})=0$ for all $i\in N$, the corresponding inequality in
(\ref{eq:bg}) yields $r(N)\geqslant 1$ because the weights for a m.b.c. being a partition are 1, and since $r'(N)=-r(N)$, we get
$r'(N)\leqslant -1$. Taking now the partition $N^\bot$, we
obtain
\[
0 -r'(N)\leqslant 0
\]
which is impossible. Letting $r(T)<0$ yields the same contradiction by inverting
the roles of $r$ and $r'$. We conclude that $r(T)=0, \forall T\ne S.$ Finally, observe that $r(S)>0$ is not possible as the inequality for
partition $\{S,(N\setminus S)^\bot\}$ would not be satisfied. This finishes the proof.

\item Let us take $i\in N$ and show that $r_i$ is an extremal ray. First, as

$$ \sum_{S\in {\cB}, S\ni i} \lambda_S^{\cB} =1, \forall {\cB}\in\gB^*(n) ,$$ it follows that $r_i$ satisfies all inequalities in (\ref{eq:bg}); moreover, $r_i(\{j\} )=0$ for all
$j\in N$ and hence, $r_i$ is a ray of $\cBG^0(n).$ To show that it is extremal, we need to show that the set of solutions
of the subsystem of (\ref{eq:bg}) formed by tight inequalities has dimension 1 in
$\cBG^0(n)$. Consider a m.b.c. $\cB$ and its corresponding inequality:
\[
\sum_{S\in\cB}\lambda_S^\cB r_i(S) - r_i(N)\leqslant 0.
\]

We obtain
\[
\sum_{S\in\cB}\lambda_S^\cB r_i(S) - r_i(N) = \sum_{\substack{S\in\cB\\S\ni
    i\\|S|>1}}\lambda_S^\cB - 1.
\]

Hence, the inequality is tight iff $\{i\}\not\in\cB$. Let us call (\ref{eq:bg})' the subsystem of tight
inequalities in $v$, with $v\in\cBG^0(n)$.
Consider the partition $\{S,(N\setminus S)^\bot\}$ with
$1<|S|<n$. In terms of (\ref{eq:bg})', the corresponding equality is $v(S)=v(N)$. Now,
consider any m.b.c. $\cB$ such that $\{i\}\not\in\cB$. The corresponding
equality reads
\[
\sum_{\substack{S\in\cB\\S\ni i}}\lambda^\cB_Sv(S)-v(N)=0.
\]

Substituting $v(S)$ by $v(N)$ yields $v(N)-v(N)=0$. Therefore, the subspace of
solutions is given by $\{(\alpha,\ldots,\alpha),\alpha\in\RR\}$, of dimension 1.

\item It remains to prove that there is no other extremal ray. Consider a ray $w$
in $\cBG^0(n)$, hence satisfying $w(\{i\})=0$ for all $i\in N$, and
(\ref{eq:bg}). Suppose $w$ is not a conic combination of the extremal rays $r_S$
and $r_i$, $S\subset N$, $|S|>1$, $i\in N$, i.e., the following system has no
solution in $\alpha_S,\alpha_i$:

\begin{equation}\label{eq:s1}
  \begin{array}{rl} \sum_{S\subset N, |S|>1}\alpha_Sr_S + \sum_{i\in N}\alpha_ir_i&= w
  %\label{eq:s1}
  \\
  \alpha_S &\geqslant 0, \quad S\subset N, |S|>1
  %\label{eq:s2}
  \\
  \alpha_i &\geqslant 0, \quad i\in N.
  %\label{eq:s3}
  \end{array}
\end{equation}

Using definitions of $r_S, r_i$ and omitting coordinates for singletons in $r_S,r_i,w$ as they are all
0, we obtain that the previous system can be written as

\begin{equation}\label{eq:t}
 \begin{array}{rl} -\alpha_S +\sum_{i\in S}\alpha_i &= w(S), \quad S\subset N,|S|>1\\
  \sum_{i\in N}\alpha_i & = w(N)
  %\label{eq:t}
  \\
    \alpha_S &\geqslant 0, \quad S\subset N, |S|>1\\
  \alpha_i &\geqslant 0, \quad i\in N.
  %\label{eq:tt}
  \end{array}
\end{equation}

We may denote (with some abuse) the
whole system (\ref{eq:t}) by $A\alpha\geqslant b$ in matrix notation.
If this system has no solution, then, by Farkas' Lemma, this is equivalent to say that there exists a vector $[y \ z\ t]$ with coordinates $y_S\in\RR$,
$S\subseteq N, |S|>1$, $z_T\geqslant 0$, $T\subset N,|T|>1$, and $t_i\geqslant
0$, $i\in N$, such that

$$[y \ z\ t]^\top A=0, \quad [y\ z\ t]^\top b>0.$$

Observe that the only vectors $[y
  \ z\ t]$ such that $[y \ z\ t]^\top A=0$ have the form:
\[
y_S=1, \quad z_S=1 \ (S\subset N, |S|>1), \quad y_N=-2^{n-1}+2-t, \quad  t_1= ... =t_n=t,
\]
up to a multiplicative factor $K>0$.

Then, we obtain
\[
[y\ z\ t]^\top b = \sum_{S\subset N,|S|>1} w(S)-(2^{n-1}-2+t) w(N).
\]

Observe that $\{S\subset N,|S|>1\}$ is a balanced collection (not minimal if
$n>3$) with balancing weights $\frac{1}{2^{n-1}-2}$. It follows that
\[
\sum_{S\subset N,|S|>1}w(S)\leqslant (2^{n-1}-2)w(N).
\]

Therefore $[y\ z\ t]^\top b\leqslant
(2^{n-1}-2)w(N)-(2^{n-1}-2+t)w(N)=-tw(N)\leqslant 0$. The last inequality
follows from $t\geqslant 0$ and (\ref{eq:t}). Hence,  this system has no
solution and consequently, system (\ref{eq:t}) has always a solution. Therefore, $w$ is
not extremal and the result follows.
\end{enumerate}
\end{proof}

Compare also this result with Lemma 5.4 in \cite{krst19}, giving condition
  for a ray to be extreme in the dual cone, and Corollary 5.1. mentioning the
  extremal rays $-\delta_S$.

The following result is immediate.
\begin{lemma}\label{lem:bg}
The cores of $w_i$, $-w_i$, $r_i$, $r_S$ for all $i\in N$, $S\subset N$, $|S|>1$
are singletons (respectively, $\{1_i\}$, $\{-1_i\}$,
$\{1_i\}$, $\{0\}$).
\end{lemma}

\subsection{Facets}
\begin{theorem}\label{th:bgfacet}
  Each inequality in (\ref{eq:bg}) defines a facet, i.e., minimal balanced
  collections in $\gB^*(n)$ are in bijection with the facets of $\cBG(n)$.
\end{theorem}
\begin{proof}
  Let us consider $\cB\in\gB^*(n)$ and the inequality
  \begin{equation}\label{eq:face}
  \sum_{S\in\cB}\lambda^\cB_Sv(S)-v(N)=0.
  \end{equation}

  It suffices to show that this face
  contains $2^n-2$ independent extremal rays. We already know by definition of
  the lineality space that any face contains the $n$ rays $w_1,\ldots,w_n$ of the
  lineality space.

  Among the extremal rays $r_S$, $S\subset N$, $|S|>1$, it is easy to check that
  only those such that $S\not\in\cB$ satisfy (\ref{eq:face}). There are
  $2^n-n-2-|\cB|+k_\cB$ such rays, where $k_\cB$ is the number of singletons in $\cB$.

  Next, define $B=N\setminus \bigcup_{\{i\}\in\cB}\{i\}$. Then $|B|=n-k_\cB$.
  Now, observe that any $r_i$ with $i\in B$ is satisfying
  (\ref{eq:face}). Indeed,
  \[
 \sum_{S\in\cB}\lambda^\cB_Sr_i(S)-r_i(N) = \sum_{\substack{S\in\cB\\S\ni
     i\\|S|>1}}\lambda_S^\cB -1 = 0.
 \]

 This makes another set of $n-k_\cB$ rays. Therefore, we have in total $2^n-2+n-|\cB|\geqslant
 2^n-2$, as $\cB$ is minimal (see Section~\ref{sec:bg}). It remains to prove
 independence. Observe that every $\delta_{\{i\}}$, $i\in N$, is used by
 $w_i$, every $\delta_S$, $|S|>1$, $S\not\in\cB$, is used by $r_S$, and every
 $\delta_S$ for  $S\in\cB$ is used in some of the $r_i$'s with $i\in B$, unless
 $S\subseteq N\setminus B$. But then consider $\cB'= \cB \setminus \{ S\} $ and weights

 \[ \lambda_T^{\cB'} = \left\{ \begin{array}{cc} \lambda_i^{\cB} +
   \lambda_S^{\cB}, & T=\{i\}, i\in S \\ \lambda_T^{\cB'}, & \text{ otherwise} \end{array} \right. ,\] and hence, $\cB$ is not minimal, a contradiction. If
 $|\cB|=n$, then we have exactly $2^n-2$ extremal rays, which are independent by
 the above argument. If $|\cB|<n$, then $n-|\cB|$ rays $r_i$ have to be
 removed to make the family independent.
\end{proof}

\section{The polyhedron $\cBG_\alpha(n)$}\label{sec:bgalpha}

Let us now study the set $\cBG_\alpha(n)$, defined by (see (\ref{eq:bgalpha})):
\[
\cBG_\alpha(n)=\Big\{v\in\cG(n)\mymid \sum_{S\in\cB}\lambda_S^\cB v(S)\leqslant \alpha,\, \forall \cB\in\gB^*(n)\Big\}.
\]

We follow the same notation as in Section~\ref{sec:bg}, except that now $v$ is a
vector in $\RR^{2^N\setminus\{\varnothing,N\}}$.  As the proof techniques
  are similar to the case of $\cBG(n)$, the proofs are relegated to the Appendix
  or omitted. We begin by expliciting the structure of $\cBG_\alpha(n)$.

\begin{theorem}\label{th:5}
Let $n\geqslant 2$, $\alpha\in\RR$. Then, $\cBG_\alpha(n)$ is a
$(2^n-2)$-dimensional affine cone\footnote{ i.e., a cone plus a point.}, which is not pointed. Its lineality space has dimension $n-1,$ with base
%It contains an $(n-1)$-dimensional affine space defined by , whose lineality space has basis
$(w_i)_{i\in
  N\setminus\{n\}}$, where
\[
w_i = \sum_{\substack{S\ni i\\S\not\ni n}}\delta_S - \sum_{\substack{S\not\ni i\\S\ni n}}\delta_S.
\]

The affine space is given by
\begin{align*}
  \sum_{i\in N}v(\{i\}) &= \alpha\\
  v(S)+\sum_{i\in N\setminus S}v(\{i\}) & = \alpha, \quad S\subset N, |S|>1.
\end{align*}
\end{theorem}

In Theorem~\ref{th:5}, element $n$ plays a particular role, but this choice is
arbitrary.
Indeed, note that the affine space in $\cBG_\alpha(n)$ contains in particular the games
$\alpha u_{\{i\}}$, for all $i\in N$.

To find the extreme rays of $\cBG_\alpha(n)$, we proceed as for
$\cBG(n)$. Using the notation from the proof of Th.~\ref{th:5}, we
write

$$\cBG_\alpha(n)=\alpha u _{\{n\}}+\cC_\alpha(n),$$ where $\cC_\alpha(n)$ is
a cone defined by the system (\ref{eq:6}). Now,

$$\cC_\alpha(n)=\Lin(\cC_\alpha(n))\oplus
\cC^0_\alpha(n),$$ where $\Lin(\cC_\alpha(n))$ is the lineality space of
$\cC_\alpha(n)$ and $\cC_\alpha^0(n)$ is its supplement where the coordinates
of the singletons $\{1\},\ldots,\{n-1\}$ are zero.

\begin{theorem}\label{th:6}
  Let $n\geqslant 2$ and $\alpha\in\RR$. The extremal rays of $\cBG_\alpha(n)$
  are:

  \begin{itemize}
  \item The $2n-2$ extremal rays corresponding to  $\Lin(\cC_\alpha(n))$:
    $w_1,\ldots,w_{n-1},-w_1,\ldots,-w_{n-1}$;
  \item  $2^n-n-2$ extremal rays of the form $r_S=-\delta_S$, $S\subset N$, $|S|>1$;
  \item $n$ extremal rays of the form
    \[
r_i = \sum_{\substack{S\ni i\\ S\not\ni n\\|S|>1}}\delta_S-\sum_{\substack{S\not\ni i\\ S\ni n}}\delta_S,\quad i\in N\setminus\{n\},
\]
and $r_n=-\delta_{\{n\}}$.
  \end{itemize}

This yields in total $2^n+2n-4$ extremal rays.
\end{theorem}

Observe that by definition of the lineality space, the extremal rays $w_i,-w_i$
belong to every facet of $\cC_\alpha(n)$. We have also the following result
(proof is omitted).

\begin{lemma}
 The core of $w_i$, $-w_i$ for all $i\in N\setminus \{n\}$, $r_i$ for all
  $i\in N$, $r_S$ for all $S\subset N$, $|S|>1$ or $S=\{ n\} ,$ are reduced to singletons,
  which are $\{1_i-1_n\}$, $-\{1_i+1_n\}$,
  $\{1_i-1_n\}$, and $\{0\}$, respectively.
\end{lemma}

\begin{theorem}\label{th:bgafacet}
  Each inequality in (\ref{eq:bgalpha}) defines a facet, i.e., minimal balanced
  collections in $\gB^*(n)$ are in bijection with the facets of $\cBG_\alpha(n)$.
\end{theorem}
The proof is similar to the one of Theorem~\ref{th:bgfacet} and is omitted.

\section{The polytope $\cBG_+(n)$}\label{sec:bg+}

Finally, let us study $\cBG_+(n)$. Recall that $\cBG_+(n)$ is defined by (see (\ref{eq:bg+})):
\[
\cBG_+(n)=\Big\{v\in\cG_+(n)\mymid \sum_{S\in\cB}\lambda_S^\cB v(S)\leqslant 1,\, \forall \cB\in\gB^*(n)\Big\}.
\]

\subsection{Dimension, boundedness}
\begin{proposition}\label{pr:1}
$\cBG_{+}(n)$ is a $(2^n-2)$-dimensional polytope.
\end{proposition}

\begin{proof}
First, note that $\cBG_{+}(n)$ is defined by a set of linear equations. Hence,
it is a polyhedron of dimension at most $2^n-2$ as $v(N)=1$ is fixed. Moreover,
the $2^n-1$ unanimity games $u_S$, $S\neq \varnothing$ belong to $\cBG_{+}(n)$
(see Example~\ref{ex:1}) and they are  affinely independent because the
  $2^n-2$ games $u_S-u_N$, $S\neq\varnothing,N$, are linearly independent, so that the
polyhedron is $2^n-2$-dimensional.

Second, the polyhedron is bounded. To see this, note that if $v\in \cBG_{+}(n),$
it follows that $v(S)\leqslant 1,\forall S\ne \varnothing , N.$ Indeed, if
$v(S)>1$, as $v$ is balanced, there exists $x\in \RR^N$ such that $x(T)\geqslant
v(T)$ for all $T\subseteq N$, and $x(N)=1$. However, as $v(\{i\})\geqslant 0$,
we have $x_i\geqslant 0$, which implies $x(S)\leqslant 1$, a contradiction.
\end{proof}

\subsection{Vertices}

\begin{theorem}\label{th:1}
Consider $v\in \cG_+(n).$ Then, $v$ is a vertex of $\cBG_+(n)$ if and only if
$v$ is balanced and 0-1-valued.
\end{theorem}

\begin{proof}
$\Leftarrow)$  Suppose $v$ is balanced and 0-1-valued. Assume $v$ is not extremal. Then, there
  exist $v',v''\in \cBG_+(n)$ such that $v=\frac{1}{2}(v'+v'')$. Take (if it
  exists) $S\neq N$ such that $v(S)=1$.
We have already shown in the proof of Proposition~\ref{pr:1} that $v'(S)>1$ is impossible.
Similarly, $v'(S)<1$ is impossible as it
  imposes $v''(S)>1$. It follows that $v'(S)=v''(S)=1$ for all $S\subseteq N$
  such that $v(S)=1$.

  Take now $\varnothing\neq S\subset N$ such that $v(S)=0$. Taking $v'(S)>0$
  forces $v''(S)<0$, which is impossible. Therefore, $v'(S)=v''(S)=0$ for all
  $S\subset N$ s.t. $v(S)=0$. We conclude that $v'=v''=v$, i.e., $v$ is an
  extreme point.

  \medskip

  \noindent
  $\Rightarrow)$
Take $v$ balanced and extremal, and suppose by contradiction that there exists
$S\subset N$ s.t. $0<v(S)<1$. We distinguish two cases. Suppose first that there
exists some core element $x\in C(v)$ such that $x(S)>v(S)$. Then, consider the
two games $v',v''$ which differ from $v$ only inasmuch as $v'(S)=v(S)-\epsilon$
and $v''(S)=v(S)+\epsilon$, with $0<\epsilon<x(S)-v(S)$. Then $v',v''$ are
balanced since $x\in C(v')$ and $x\in C(v'')$. Therefore, $v$ is not extremal
as $v=\frac{1}{2}(v'+v'')$.

Suppose now that no such core element exists, i.e., $x(S)=v(S)$ for all $x\in
C(v)$. As $x(S)=v(S)<1,$ there exists $j\in N\setminus S$ such that $x_j>0$.
Similarly, as $0<v(S)=x(S),$ we conclude that there exists $i\in S$ such that
$x_i>0.$ We define the two games $v',v''$ by
\begin{align*}
v'(T) &=v(T)+\epsilon, \quad v''(T)=v(T)-\epsilon, \quad \forall T \text{ such
  that } i\in T,j\not\in T\\
v'(T) &=v(T)-\epsilon, \quad v''(T)=v(T)+\epsilon, \quad \forall T \text{ such
  that } j\in T,i\not\in T \\
v'(T) &=v(T), \quad v''(T)=v(T), \quad \text{ otherwise },
\end{align*}
with $\epsilon>0$ small enough so that $v'(T), v''(T)\in [0,1]$, and $0<x_i-\epsilon<x_i+\epsilon<1$, $0<x_j-\epsilon<x_j+\epsilon<1$.
Clearly, $v=\frac{1}{2}(v'+v'')$. Observe that
$v',v''$ are balanced, as $x',x''\in \RR^N$ defined by
\[
x'_i=x_i+\epsilon, x'_j=x_j-\epsilon, \quad x''_i=x_i-\epsilon, x''_j=x_j+\epsilon , \quad x'_k= x''_k = x_k, \forall k\ne i, j,
\]
 are core elements of $v'$ and $v''$, respectively.
Hence, $v$ is not extremal.
\end{proof}

Hence, a vertex $v$ of $\cBG_+(n)$ is characterized in terms of the subsets
$S\in 2^N\setminus\{\varnothing,N\}$ such that $v(S)=1$. Let us denote the family
of such subsets by $\cD.$ In next result we treat the
reciprocal problem, i.e., we give the necessary and sufficient conditions for a family of subsets $\cD$ in
$2^N\setminus \{ \varnothing , N\} $ to determine a vertex of $\cBG_+(n).$

\begin{theorem}\label{th:2}
Let $\cD$ be a family of subsets $\cD$ in $2^N\setminus \{ \varnothing , N\} .$ Then, $\cD$ defines a vertex of $\cBG_+(n)$ iff either $\cD =\varnothing $ or $\bigcap \cD\neq\varnothing .$
\end{theorem}

\begin{proof}
The case $\cD =\varnothing $ defines $v(S)=0,\forall S\subset N$, which is
clearly balanced. For the rest of the proof we consider $\bigcap
\cD\neq\varnothing $.

Choose some family $\cD$ s.t. $\bigcap \cD\ni i$ for some $i\in N$, and construct
the corresponding $v$. By definition, $v$ is 0-1-valued. Hence, by
Theorem~\ref{th:1} we just have to check balancedness. Take $x\in \RR^N$
s.t. $x_i=1$, and $x_j=0$ for $j\neq i$. Then $x\in C(v)$.

Conversely, suppose $\cD\neq\varnothing$ and $\bigcap\cD=\varnothing$. We show
that the game $v$ corresponding to $\cD$ is not balanced. Consider a maximal
family $B_1,\ldots, B_r\in \cD$ s.t. $B_1\cap \cdots \cap B_r=T\ne \varnothing
$. Suppose $C(v)\neq\varnothing$ and take a core element $x\in C(v)$.
As $x(B_i)\geqslant v(B_i)=1$, $x\geqslant 0$ and $x(N)=1$, it follows that
$x(B_i)=1$ for $i=1,\ldots,r$. Therefore
\[
x(B_i\setminus T) = 1- x(T), \quad i=1,\ldots, r.
\]
Summing up we find
\[
 (r-1) \sum_{i\not\in T} x_i \geqslant \sum_{i=1}^r x(B_i\setminus T) = r(1-x(T)),
\]
where the first inequality comes from the fact that each $i\in N\setminus T$
belongs at most to $r-1$ sets in the family $B_1\setminus T,\ldots, B_r\setminus
T$. If $r>1$, we obtain
\[
\sum_{i\not\in T}x_i\geqslant \frac{r}{r-1}(1-x(T))>1-x(T),
\]
which implies $x(N)>1$, a contradiction. Therefore, it must be that $r=1$, i.e.,
there exist $S,S'\in\cD$ such that $S\cap S'=\varnothing$. As $v(S)=v(S')=1$, a
core element $x$ should satisfy $x(S)=x(S')=1$, which implies $x(N)\geqslant 2$,
a contradiction. As a conclusion, no core element $x$ exists.
\end{proof}

Note that the vertex corresponding to the empty collection is $u_N$, the
unanimity game centered on $N$ (equivalently, the Dirac game $\delta_N$).

The next result gives explicitely the core of each vertex.
\begin{proposition}\label{prop:v}
Let $v$ be a vertex of $\cBG_+(n)$, with associate collection $\cD$. If $\cD$ is
not the empty collection, then
\begin{equation}\label{eq:cv}
C(v) = \conv\Big\{\1^{\{i\}}\mymid i\in\bigcap\cD\Big\},
\end{equation}
which implies that the dimension of the core of $v$ is $|\bigcap\cD|-1$. If
$\cD$ is the empty collection, then $v=u_N$, whose core is the simplex
$\Delta(n):=\{x\in\RR^N_+\mymid \sum_{i\in N}x_i=1\}$.
\end{proposition}
\begin{proof}
$x$ is a core element iff $x(\bigcap\cD)=1$, from which the result follows.
\end{proof}

Proposition~\ref{prop:v} generalizes some known results for simple games. Recall
that a {\it simple game}  $v$ is a 0-1-valued game which is monotonic, i.e.,
$S\subseteq T$ implies $v(S)\leqslant v(T)$. The collection $\cD$ is called the
collection of {\it winning coalitions}, and  $\bigcap \cD$ is the
set of {\it veto players}. It is well known that the core of a simple game is empty if
and only if there is no veto player, and when nonempty, it is expressed by
(\ref{eq:cv}), see, e.g., \cite{pet08}. Our result is more general as not all
vertices of $\cBG_+(n)$ are simple games. We can also deduce the following
(a simple game is {\it proper}  if $S\in\cD$ implies $N\setminus S\not\in\cD$):
\begin{corollary}
If a voting game is balanced, then it is proper.
\end{corollary}
\begin{proof}
Suppose there exists a balanced voting game $v$ that is not proper. Then, there exists $S$ s.t. $S, N\setminus S\in \cD.$ But then, $\bigcap \cD =\emptyset, $ contradicting Theorem~\ref{th:2}.
\end{proof}

\subsection{Vertex enumeration and generation}
We deduce from Theorem~\ref{th:2} that the enumeration of vertices amounts to
the enumeration of the collections $\cD$ whose intersection is nonempty. In the
next result we obtain a recursive formula to compute the number of vertices of
$\cBG_{+}(n)$, which we denote by $b_n.$ Also, for further use, we introduce a
number of notations. We denote by $\cA_N$ the set of all collections of sets in
$2^N \setminus \lbrace \varnothing, N \rbrace$, including the empty collection. The cardinality of
$\cA_N$ is $t_n=2^{2^n-2}.$ We introduce also $\cF_N$ and $\cS_N$ the set of all
{\it nonempty} collections
in $\cA_N$ with a nonempty intersection, and empty intersection,
respectively. Their cardinalities are denoted by $f_n$ and $s_n$,
respectively. We have by definition $t_n=f_n+s_n+1$.

\begin{theorem}\label{thm3}
The number of vertices $b_n$ of $\cBG_{+}(n)$ is given by $b_n=f_n + 1$ where $f_n$ can be obtained recursively as follows:

$$f_n=\sum_{k=1}^{n-1} {n \choose k} \left( 2^{2^{k}-1} - f_{k} -1\right), \forall n>1 \text{ and } f_1=0.$$
\end{theorem}

\begin{proof}
 Vertices of $\cBG_{+}(n)$ are in bijection with
collections $\cD \in \cA_N$ such that $\bigcap \cD \neq \varnothing$,
plus the empty collection (associated to $u_N$).

Now, observe that $f_n$ can be split by considering the different choices for
$\bigcap \cD.$ Suppose that the intersection is the set $S,$ that is, $\bigcap
\cD :=S,$ with $|S|=k.$ There are two possibilities: either $S\in \cD$ or $S
\notin \cD.$ If $S\notin \cD,$ then $\cD$ is associated to a nonempty
collection $\cB$ given by $\cB =\{ B\setminus S: B\in \cD \} $. Observe that
$\cB \in \cS_{N\setminus S}$.  If $S\in \cD ,$ then $\cD$ is
associated to a collection $\cB \in \cA_{N\setminus S}$ given by $\cB =\{
B\setminus S: B\in \cD \}$ if $\cD\neq\{S\}$, and $\cB=\varnothing$
otherwise. This way, summing over all $S$, we obtain:

\[
f_n=\sum_{k=1}^{n-1} {n \choose k} \left( s_{n-k} + t_{n-k} \right) = \sum_{k=1}^{n-1} {n \choose k} \left( 2t_{n-k} - f_{n-k} -1\right) = \sum_{k=1}^{n-1} {n \choose k} \left( 2^{2^{n-k}-1} - f_{n-k} -1\right).
\]

Making the change $k'=n-k$ and using ${n \choose k}={n \choose n-k},$ the result holds.
\end{proof}

The first values of $b_n$ can be seen in Table \ref{Tabla_vert_bg}. As we can see, the number of vertices grows exponentially.

\begin{table}[h]
\begin{center}
\begin{tabular}{|c|cccccccccc|}
\hline $n$ & $1$ & $2$ & $3$ & $4$ & $5$ & $6$ & $7$ & $8$ & $9$ & $10$\\

\hline $b_n$ & $1$ & $3$ & $19$ & $471$ & $162631$ & $12884412819$ & $6.456e19$ & $1.361e39$ & $5.210e77$ & $6.703e154$\\

\hline
\end{tabular}
\caption{Number of vertices of $\cBG_{+}(n)$}
\label{Tabla_vert_bg}
\end{center}
\end{table}

\begin{remark}
The integer sequence of the number of vertices given in
Table~\ref{Tabla_vert_bg} happens to be already known: it
appears as sequence A051381 in the On-line Encyclopedia of Integer Sequences
(OEIS) \cite{oeis}, and is referred to as the ``number of Boolean functions of
$n$ variables from Post class $F(5,inf)$''. The founding paper is by Jojovi\'c
and Kilibarda \cite{joki99}. It gives an explicit (non-recursive) formula for
this sequence, shown in a very general context and with a long proof. We have
kept our recursive formula with its short proof, as we need it hereafter for random
vertex generation.
\end{remark}

Table \ref{vert_n=3} enumerates the 19 vertices for $n=3.$ Vertices for $n=4$
are available as supplementary material. For simplicity, braces and
commas are omitted for writing sets, e.g., 12 instead of $\{1,2\}$, etc. Also, we have used the following notation: for any
collection $\cC\subseteq 2^N\setminus\{\varnothing,N\}$, we define the game
\[
d_\cC:=\sum_{S\in\cC}\delta_S + \delta_N,
\]
e.g., $d_{1,12}=\delta_1+\delta_{12}+\delta_{123}$.

\begin{table}[h]
\begin{center}
\begin{tabular}{|c|cccccc|}
\hline Vertices  & 1 & 2 & 3 & 12 & 13 & 23 \\
\hline $u_{123}$ & 0 & 0 & 0 & 0   & 0   & 0   \\
       $u_1$     & 1 & 0 & 0 & 1   & 1   & 0   \\
       $u_2$     & 0 & 1 & 0 & 1   & 0   & 1   \\
       $u_3$     & 0 & 0 & 1 & 0   & 1   & 1   \\
       $u_{12}$  & 0 & 0 & 0 & 1   & 0   & 0   \\
       $u_{13}$  & 0 & 0 & 0 & 0   & 1   & 0   \\
       $u_{23}$  & 0 & 0 & 0 & 0   & 0   & 1   \\
$u_{12}\vee u_{13}$& 0 & 0 & 0 & 1   & 1   & 0   \\
$u_{12}\vee u_{23}$& 0 & 0 & 0 & 1   & 0   & 1   \\
$u_{13}\vee u_{23}$& 0 & 0 & 0 & 0   & 1   & 1   \\
       $d_1$     & 1 & 0 & 0 & 0   & 0   & 0   \\
       $d_2$     & 0 & 1 & 0 & 0   & 0   & 0   \\
       $d_3$     & 0 & 0 & 1 & 0   & 0   & 0   \\
       $d_{1,12}$& 1 & 0 & 0 & 1   & 0   & 0   \\
       $d_{1,13}$& 1 & 0 & 0 & 0   & 1   & 0   \\
       $d_{2,12}$& 0 & 1 & 0 & 1   & 0   & 0   \\
       $d_{2,23}$& 0 & 1 & 0 & 0   & 0   & 1   \\
       $d_{3,13}$& 0 & 0 & 1 & 0   & 1   & 0   \\
       $d_{3,23}$& 0 & 0 & 1 & 0   & 0   & 1   \\
\hline
\end{tabular}
\caption{List of vertices of $\cBG_{+}(n)$ for $n=3$. ``$\vee''$ indicates the
  maximum.}
\label{vert_n=3}
\end{center}
\end{table}

Based on the result of Theorem \ref{thm3} and its proof, we develop an algorithm that
generates vertices of $\cBG_{+}(n)$ uniformly at random. For this,
consider the recursive expression (see proof of Theorem~\ref{thm3})
\begin{equation}\label{recursive}
  f_n=\sum_{k=1}^{n-1} {n \choose k} \left( s_{n-k} + t_{n-k} \right) .
\end{equation}
This formula comes from the following fact. Suppose $\bigcap\cD=S$, with
$|S|=k$. If $S\in \cD$, removing $S$ from $\cD$ and from each set in $\cD$
yields a collection in $\cA_{N\setminus S}$, while if $S\not\in\cD$, removing
$S$ from each set in $\cD$ yields a collection in $\cS_{N\setminus S}$.
Note that generating a collection $\cD$ randomly in $\cA_N$ is very
simple. It suffices to sort all $2^n-2$ sets (for example lexicographically),
and generate a vector of $2^n-2$ zeros and ones, the ones marking the sets that
are in $\cD.$

The algorithm goes as follows. First, we consider the possibility of
drawing the vertex $u_N$. As $b_n=f_n+1,$ we choose $u_N$ with
probability

$$p_0(n) := \dfrac{1}{1+f_n}.$$

Table \ref{Tabla_p0} shows the first values for $p_0(n).$ As it can be seen, this value is almost zero for $n> 4.$

\begin{table}[h]
\begin{center}
\begin{tabular}{|c|ccccccccc|}
\hline $n$ & $1$ & $2$ & $3$ & $4$ & $5$ & $6$ & $7$ & $8$ & $9$ \\

\hline $p_0(n)$ &  $1$ & $0.333$ & $0.052$ & $0.002 $ & $6.14e-06$ & $7.76e-11$ & $1.54e-20$ & $7.34e-40$ & $1.91e-78$ \\

\hline
\end{tabular}
\caption{First values of $p_0(n)$}
\label{Tabla_p0}
\end{center}
\end{table}

If $u_N$ is not chosen, then we must generate a random element $\cD$ of
$\cF_N.$ Using (\ref{recursive}), we start by choosing the number of elements in $\bigcap \cD.$
The probability of $|\bigcap \cD |=k$ is given by

$$p^k_1(n)=\dfrac{{n \choose k} \left( s_{n-k} + t_{n-k} \right)}{f_n}.$$

Table \ref{Tabla_p1} shows the first values of $p^k_1(n).$

\begin{table}[h]
\begin{center}
\begin{tabular}{|c|cccccc|}
\hline $k$ & $1$ & $2$ & $3$ & $4$ & $5$ & $6$ \\

\hline $p_1^k(2)$ & $ 1$ & $ $ & $ $ & $ $ & $ $ & $ $  \\

\hline $p_1^k(3)$ & $0.83$ & $0.17$ & $ $ & $ $ & $ $ & $ $   \\

\hline $p_1^k(4)$ & $0.93$ & $ 0.06$ & $0.01 $ & $ $ & $ $ & $ $\\

\hline $p_1^k(5)$ & $0.993 $ & $0.006 $ & $0.001 $ & $ \approx 0$ & $ $ & $ $  \\

\hline $p_1^k(6)$ & $0.999 $ & $ \approx 0$ & $\approx 0 $ & $ \approx 0$ & $\approx 0 $ & $ $  \\

\hline $p_1^k(7)$ & $ \approx 1$ & $ \approx 0$ & $\approx 0 $ & $\approx 0 $ & $\approx 0 $ & $ \approx 0$  \\

\hline
\end{tabular}
\caption{First values of $p^k_1(n)$}
\label{Tabla_p1}
\end{center}
\end{table}

As we can see, by far the most likely is that the intersection of the sets of
$\cD$ has cardinal 1. For $n\geq 5,$ we may consider the probability
distribution $p^k_1(n)$ to be approximately the Dirac distribution at $k=1$.

The next step in the algorithm is to choose a set $S$ of cardinal $k$ at random
among the ${n \choose k}$ possibilities, then to decide if either $S\in \cD$ or $S\notin \cD$. For a given selected $k,$ the  probability of $S\in \cD$ is given by

$$p_2(n-k)=\dfrac{ t_{n-k}}{s_{n-k} + t_{n-k}}.$$

Table \ref{Tabla_p2} gives the first values of $p_2(n).$

\begin{table}[h]
\begin{center}
\begin{tabular}{|c|ccccccccc|}
\hline $n$ & $1$ & $2$ & $3$ & $4$ & $5$ & $6$ & $7$ & $8$ & $9$ \\

\hline $p_2(n)$ &  $1$ & $0.80$ & $0.58$ & $0.51 $ & $0.50$ & $0.50$ & $0.50$ & $0.50$ & $0.50$ \\

\hline
\end{tabular}
\caption{First values of $p_2(n)$}
\label{Tabla_p2}
\end{center}
\end{table}

As we can see, for $n\geq 5,$ the probabilities of whether the set $S$ is
in the selected collection $\cD$ or not are approximately the same.
%Then, with probability $p_2(n-k)$ the set $S$ is in the collection $\cD$ and with probability $1-p_2(n-k)$ the set $S$ is not in $F$.
If $S\in \cD,$ it suffices to generate uniformly an element $\cA$ in
$\cA_{N\setminus S}$ with the procedure described above and the final collection will be

$$\cD=\lbrace  S \rbrace \bigcup_{A \in \cA} \lbrace S \cup A \rbrace .$$

If the set $S\notin \cD$, it suffices to generate a random element $\cS$ of
$\cS_{N\setminus S}.$ To achieve this task, we use  the fact that the quotient $s_n/t_n$ is almost 1 for $n\geq 5,$ as it can be seen in Table \ref{Tabla_s_t}.

\begin{table}[h]
\begin{center}
\begin{tabular}{|c|ccccccccc|}
\hline $n$ & $1$ & $2$ & $3$ & $4$ & $5$ & $6$ & $7$ & $8$ & $9$ \\

\hline $s_n/t_n$ &  $0$ & $0.25$ & $0.70$ & $0.97 $ & $0.99$ & $\approx 1$ & $\approx 1$ & $\approx 1$ & $\approx 1$ \\

\hline
\end{tabular}
\caption{First values of $s_n/t_n$}
\label{Tabla_s_t}
\end{center}
\end{table}

Hence, for $n\geq 3,$ an element $\cS$ in $\cS_N$ can be generated uniformly at
random via a rejection sampler on $\cA_N.$ That is, we generate an element $\cA$
of $\cA_N.$ If $\cA \in \cS_N,$ then we return $\cA$. Otherwise $\cA
\notin \cS_N$ and we repeat the procedure until we obtain an element in
$\cS_N.$ For $n\leq 2,$ we have $s_1=0$ and $s_2=1,$ so these cases are
trivial. Once $\cA\in \cS_N$ is generated, the chosen collection is

$$\cD =\bigcup_{A \in \cA} \lbrace S \cup A \rbrace.$$

We give below a summary of this algorithm.
\begin{algorithm}[h]
\begin{algorithmic}
\caption{VERTEX SAMPLER FOR $\cBG_{+}(n)$ }\label{algor_vertex_sampler}
\State \textbf{Step 1:} With probability $p_0(n)$ \textbf{return} $\cD
=\varnothing$ (vertex $u_N$). With probability $1-p_0(n)$ go to Step 2.
\State \textbf{Step 2:} Choose some $k\in \lbrace 1, \ldots, n-1 \rbrace$ with probability $p_1^k(n)$. Also, generate a set $S\subset N$ with $k$ elements. Go to Step 3.
\State \textbf{Step 3:} With probability $p_2(n-k)$ go to Step 4 and with probability $1-p_2(n-k)$ to Step 5.
\State \textbf{Step 4:} Generate at random some $\cA \in \cA_{N\setminus S}.$ Then \textbf{return} $\cD$ such that:

$$\cD =\lbrace  S \rbrace \bigcup_{A \in \cA} \lbrace S \cup A \rbrace.$$

\State \textbf{Step 5:} Generate at random (by rejection) some $\cA \in
\cS_{N\setminus S}.$ Then \textbf{return} $\cD$ such that:

$$\cD=\bigcup_{A \in \cA} \lbrace S \cup A \rbrace.$$
\end{algorithmic}
 \end{algorithm}

\begin{lemma}
The previous algorithm generates a vertex of $\cBG_+(n)$ uniformly at random.
\end{lemma}

\begin{proof}
It suffices to show that any family $\cD$ defining a vertex of $\cBG_+(n)$ has probability $1/b_n.$ Suppose $S=\bigcap \cD$ and $S\in \cD.$ If $|S|=k,$ the probability of selecting $\cD$ in the algoritm is given by

$$P(\cD) = \dfrac{f_n}{1+f_n} \cdot \dfrac{{n \choose k} \left( s_{n-k} + t_{n-k} \right)}{f_n} \cdot \dfrac{1}{{n \choose k}} \cdot \dfrac{t_{n-k}}{s_{n-k} + t_{n-k}} \cdot \dfrac{1}{t_{n-k}} = \dfrac{1}{1+f_n}=\dfrac{1}{b_n}.$$

Similarly if $S\notin \cD$ we get:

$$P(\cD)=\dfrac{f_n}{1+f_n} \cdot \dfrac{{n \choose k} \left( s_{n-k} + t_{n-k} \right)}{f_n} \cdot \dfrac{1}{{n \choose k}} \cdot \dfrac{s_{n-k}}{s_{n-k} + t_{n-k}} \cdot \dfrac{1}{s_{n-k}} = \dfrac{1}{1+f_n}=\dfrac{1}{b_n}.$$

As the probability of selecting the empty family is $1/b_n,$ the result follows.
\end{proof}

Let us finally deal with the computational complexity of generating a vertex according to the previous algorithm. First, let us establish a result about the growth of the number of vertices in terms of $n$.

\begin{proposition}
The asymptotic growth rate of the number of vertices, $b_n$, in $\cBG_{+}(n)$ is:
\begin{equation}\label{eq:asyn_vert}
b_n=O(n2^{2^{n-1}}).
\end{equation}
\end{proposition}

\begin{proof}
As we know $b_n=f_n+1$ where $f_n$ can be written as:

$$f_n=\sum_{k=1}^{n-1} {n \choose k} \left( s_{n-k} + t_{n-k} \right) .$$
Since $s_k \leq t_k$ we get:

$$\sum_{k=1}^{n-1} {n \choose k}  t_{n-k} \leq f_n \leq 2\sum_{k=1}^{n-1} {n \choose k}  t_{n-k}.$$

Therefore,

$$b_n=O\left( \sum_{k=1}^{n-1} {n \choose k}  t_{n-k} \right).$$

Considering that $t_k=2^{2^k-2}$ the dominant term in the expression $\sum_{k=1}^{n-1} {n \choose k}  t_{n-k}$ is the one associated with $k=n-1$ since the binomial coefficients are polynomial. Thus,

 $$b_n=O\left( n2^{2^{n-1}-2} \right)=O\left( n2^{2^{n-1}} \right).$$
\end{proof}

The following result refers to the computational complexity of the worst-case computation time required to generate a vertex by the previous algorithm.

\begin{proposition}
The computational complexity of Algorithm 1 is $O(2^n).$
\end{proposition}

\begin{proof}
As usual, we will assume that generating a random number has a complexity of $O(1).$
We also assume that the probabilities $p_0(n), p_1^k(n), p_2(n)$ are known a priori. With these considerations, we know that steps 1 and 3 have a complexity of $O(1)$ since only a random number needs to be generated. For step 2, we can choose $k$ with complexity $O(n)$, and we can generate the set $S$ with complexity $O(n)$ as well. For step 4, we can generate an element of $\mathcal{A}_{N\setminus S}$ by listing all possible subsets and generating a 0 or 1 for each case, which means a complexity of $O(2^n).$ Finally, for step 5, as the rejection rate tends to zero very quickly, the rejection rate will only affect small
$n$, in which case it would multiply the computational complexity by a constant, $O(C2^n).$ Taking all this into account, the total computational complexity would be $O(2^n).$
\end{proof}

We can see that in relation to the asymptotic number of vertices (\ref{eq:asyn_vert}), the complexity is very reduced.

\subsection{Adjacency of vertices and related properties}
Recall that two vertices are adjacent if they both belong to the same edge
(1-dimensional face of the polytope). The aim of this section is twofold: first, we will see that $\cBG_+(n)$ is a combinatorial polytope and, as a consequence, that there exists a Hamiltonian path connecting each pair of vertices. Second, we will characterize
adjacency of vertices corresponding to collections $\cD$ such that their
intersection $\bigcap\cD$ is a singleton, which is by far the most common case,
according to Table~\ref{Tabla_p1}. In addition, these vertices have a core
reduced to a singleton. This characterization is given in
Theorems~\ref{th:a1} to \ref{th:a3}.
%In the last part of this section, we prove that $\cBG_+(n)$ is a combinatorial polytope.

%we prove that this graph is Hamilton-connected.
%Recall that a Hamiltonian path in a graph is a path that visits each node of the
%graph exactly once, and a graph is {\it Hamilton-connected} if every pair of
%distinct nodes is joined by a Hamiltonian path.

In order to prove that $\cBG_+(n)$ is Hamilton-connected, we first show that $\cBG_+(n)$ is a combinatorial polytope.
\begin{definition}\cite{napu81}
A polytope $\cP$ is said to be {\it combinatorial} if the two following conditions hold:

\begin{itemize}
\item All vertices of $\cP$ are 0-1-valued.
\item Given two vertices $v_1, v_2$ of $\cP$, if they are not adjacent, then there exist two other different vertices $v_3, v_4$ such that
$$ v_1 + v_2 = v_3 + v_4.$$
\end{itemize}
\end{definition}

For a combinatorial polytope, the following can be shown.

\begin{theorem}\cite{napu81}\label{Hamilton}
Let $G$ be the adjacency graph of a combinatorial polytope. Then $G$ is
either a hypercube or is Hamilton-connected.
\end{theorem}

Now, the following holds.

\begin{lemma}\label{lem:a1}
Let $v_1,v_2$ be two distinct vertices of $\cBG_+(n)$, with associate
collections $\cD_1,\cD_2$. Then
$v_1,v_2$ are not adjacent if and only if there exist vertices $v_3,v_4$
distinct from $v_1,v_2$ such that
\[
v_1+v_2=v_3+v_4,
\]
In addition, the associated collections $\cD_3,\cD_4$ satisfy
\[
\cD_3\cup\cD_4=\cD_1\cup\cD_2, \quad \cD_3\cap\cD_4=\cD_1\cap\cD_2.
\]
\end{lemma}

\begin{proof}
Let us consider two vertices $v_1,v_2$ with collections $\cD_1,\cD_2$. They are
not adjacent if and only if there exist $\lambda,\lambda'\in[0,1]$ and
  vertices $v_3,v_4$ distinct from $v_1,v_2$ such that
  \[
\lambda v_1 +(1-\lambda)v_2 = \lambda' v_3 + (1-\lambda') v_4=:v.
\]
Let us denote by $\cD_3,\cD_4$ the collections associated to $v_3,v_4$. We have the
following, using the decomposition into $v_1,v_2$:
\[
v(S) = \begin{cases}
  \lambda, & S\in\cD_1\setminus\cD_2\\
  1, & S\in\cD_1\cap \cD_2\\
  1-\lambda, & S\in\cD_2\setminus\cD_1\\
  0, & S\not\in\cD_1\cup\cD_2.
  \end{cases}
\]
Doing similarly with $v_3,v_4$, we deduce that $\cD_1\cup\cD_2=\cD_3\cup\cD_4$,
$\cD_1\cap\cD_2=\cD_3\cap\cD_4$. As $\cD_3,\cD_4$ are distinct from $\cD_1,\cD_2$, this
imposes that:
\begin{enumerate}
\item If $\cD_1\setminus\cD_2\neq\varnothing$ and
  $\cD_2\setminus\cD_1\neq\varnothing$, then either (a) $\cD_3=\cD_1\cup\cD_2$ and
$\cD_4=\cD_1\cap\cD_2$ (or the converse), or (b) at least one of these
  collections must intersect $\cD_3$ and $\cD_4$.
\item If $\cD_1\subset \cD_2$ then
  $\cD_2\setminus \cD_1$ must intersect $\cD_3$ and $\cD_4$;
\item  Similar when
  $\cD_2\subset \cD_1$.
\end{enumerate}
In any case, this implies
$\lambda=\lambda'=1-\lambda'=1-\lambda$, i.e., $\lambda=\lambda'=1/2$.
\end{proof}

Hence, as a consequence of Theorem \ref{th:1} and Lemma \ref{lem:a1} the following holds.

\begin{theorem}\label{path}
The polytope $\cBG_+(n)$ is combinatorial. Moreover, the adjacency graph of $\cBG_+(n)$ is Hamilton-connected.
\end{theorem}

%\begin{corollary}
%The polytope $\cBG_+(n)$ is combinatorial.
%\end{corollary}
%
%From this result, we get:
%\begin{corollary}
%  The adjacency graph of $\cBG_+(n)$ is Hamilton-connected.
%\end{corollary}

\begin{proof}
$\cBG_+(n)$ is combinatorial as a straight consequence of Theorem \ref{th:1} and Lemma \ref{lem:a1}. Now, fron Theorem \ref{Hamilton}, the adjacency graph of this polytope is either Hamilton-connected or a hypercube. For $n=1$, $\cBG_+(1)$ is reduced to a singleton, therefore the result holds
trivially.
%Otherwise, it suffices by Theorem~\ref{Hamilton} to prove that the graph is not a hypercube.
For $n=2$, the number of vertices is 3, and therefore $\cBG_+(2)$ is
not a hypercube. For $n\geqslant 3$, observe that for any distinct $i,j\in N$,
$\hat{\delta}_{\{i\}}:=\delta_{\{i\}}+\delta_N$ and $\hat{\delta}_{\{j\}}$ are
vertices of $\cBG_+(n)$, however the game
$\delta_{\{i\}}+\delta_{\{j\}}+\delta_N$ is not a vertex of $\cBG_+(n)$.
\end{proof}

Let us now look more closely to adjacency. For further use, we illustrate on Figure~\ref{fig:na} the condition of
non-adjacency when $\cD_1\setminus\cD_2\neq\varnothing$ and
$\cD_2\setminus\cD_1\neq\varnothing$.

\begin{figure}[htb]
  \begin{center}
    \psset{unit=0.8cm}
    \pspicture(0,-1)(14.5,4)
    \pscircle(1.5,1.5){1.5}
    \pscircle(3,1.5){1.5}
    \psline[linestyle=dashed](0,1.5)(1.5,1.5)
    \psline[linestyle=dashed](3,1.5)(4.5,1.5)
    \uput[90](1.5,3){$\cD_1$}
    \uput[90](3,3){$\cD_2$}
    \uput[90](1,1.5){$\cB_1$}
    \uput[-90](1,1.5){$\cB_2$}
    \uput[90](3.5,1.5){$\cB_3$}
    \uput[-90](3.5,1.5){$\cB_4$}
    \rput(2.25,-1){(a)}

    \pscircle(6.5,1.5){1.5}
    \pscircle(8,1.5){1.5}
    \psline[linestyle=dashed](8,1.5)(9.5,1.5)
    \uput[90](6.5,3){$\cD_1$}
    \uput[90](8,3){$\cD_2$}
    \uput[90](8.5,1.5){$\cB_3$}
    \uput[-90](8.5,1.5){$\cB_4$}
    \rput(7.25,-1){(b)}

        \pscircle(11.5,1.5){1.5}
    \pscircle(13,1.5){1.5}
    \uput[90](11.5,3){$\cD_1$}
    \uput[90](13,3){$\cD_2$}
    \rput(12.25,1.5){$\cD_4$}
    \rput(12.25,-1){(c)}

     \endpspicture
  \end{center}
  \caption{Non-adjacency of $v_1,v_2$, with associated collections
    $\cD_1,\cD_2$. Case (a): $\cD_3=\cB_1\cup(\cD_1\cap \cD_2)\cup\cB_3$,
    $\cD_4=\cB_2\cup(\cD_1\cap \cD_2)\cup\cB_4$; Case (b):
    $\cD_3=\cD_1\cup\cB_3$, $\cD_4=(\cD_1\cap\cD_2)\cup\cB_4$ (similar when
    $\cD_1,\cD_2$ exchanged); Case (c): $\cD_3=\cD_1\cup\cD_2$, $\cD_4=\cD_1\cap\cD_2$.}
  \label{fig:na}
\end{figure}

We now focus on the vertices whose associated collections have intersection
reduced to a singleton.

\begin{theorem}\label{th:a1}
Consider two vertices $v_1,v_2$ of $\cBG_+(n)$, associated to $\cD_1,\cD_2$
respectively, and $\bigcap\cD_1=\{i\}=\bigcap\cD_2$. Then $v_1$ and $v_2$ are
adjacent iff either $\cD_1\subseteq \cD_2$ or the converse, and
$|\cD_1\Delta\cD_2|=1$.
\end{theorem}

\begin{proof}
Consider two vertices $v_1,v_2$ as above and
$v=\frac{1}{2}(v_1+v_2)$. By definition, $v(S)=1$ iff $S\in\cD_1\cap \cD_2$,
  $v(S)=1/2$ iff $S\in\cD_1\Delta\cD_2$, and $v(S)=0$ otherwise.

$\Rightarrow)$ Assume $v_1,v_2$ are adjacent and suppose
$\cD_1\setminus\cD_2\ni S_1$ and $\cD_2\setminus \cD_1\ni S_2$. Consider $v_3$
generated by $\cD_1\cup\{S_2\}$ and $v_4$ generated by
$\cD_2\setminus\{S_2\}$. Then $v_3,v_4$ differ from $v_1,v_2$ and
$v=\frac{1}{2}(v_3+v_4)$, contradicting that $v_1,v_2$ are
adjacent. Consequently, $\cD_1\subseteq \cD_2$ (or the converse).

Assuming the former, let us prove that $|\cD_2\setminus \cD_1|=1$. Suppose on
the contrary that there exist distinct $S_1,S_2\subseteq N$ s.t. $S_1,S_2\in
\cD_2\setminus\cD_1$. Hence, we can consider the games $v_3,v_4$ generated by
$\cD_1\cup\{S_1\}$ and $\cD_2\setminus\{S_1\}$ respectively. Thus defined, $v_3,v_4$
differ from $v_1,v_2$ and
$v=\frac{1}{2}(v_3+v_4)$, a contradiction.

$\Leftarrow)$ Suppose by contradiction that $v_1,v_2$ are not adjacent. Then by
Lemma~\ref{lem:a1} there exist vertices $v_3,v_4\in \cBG_+(n)$ different from
$v_1,v_2$ such that $v_1+v_2=v_3+v_4$, with associated collections
$\cD_3,\cD_4$ satisfying $\cD_1\cap\cD_2\subseteq\cD_3\subseteq \cD_1\cup\cD_2$
and similarly for $\cD_4$.

If $\cD_1\not\subseteq \cD_2$ and $\cD_2\not\subseteq \cD_1$ we are
done considering $\cD_3 = \cD_1\cap \cD_2, \cD_4= \cD_1 \cup \cD_2$. Otherwise, assume $\cD_1\subseteq \cD_2$. The above constraints resume to
$\cD_1\subseteq \cD_3\subseteq \cD_2$ and the same for $\cD_4$. As $v_3,v_4$
differ from $v_1,v_2$, strict inclusion must hold throughout, which implies
$|\cD_1\Delta\cD_2|>1$. The case $\cD_2\subseteq \cD_1$ is similar. Hence, the
result holds.
\end{proof}

\begin{lemma}\label{lem:a2}
Let $v_1,v_2$ be two vertices of $\cBG_+(n)$ with associated collections
$\cD_1,\cD_2$ such that $\bigcap\cD_1=\{i\}$, $\bigcap\cD_2=\{j\}$. If $i\neq
j$, then $\cD_1\setminus \cD_2\neq\varnothing$, $\cD_2\setminus
\cD_1\neq\varnothing$.
\end{lemma}

\begin{proof}
Suppose $i\neq j$. This implies that in $\cD_1$ there must exist $S$ such that $S\not\ni j$
(otherwise we would have $\bigcap\cD_1\supseteq \{i,j\}$). Hence $S\in \cD_1 \backslash \cD_2$. Similarly, there must
exist $T\in\cD_2$ such that $T\not\ni i$, which implies $T\in \cD_2\backslash \cD_1$.
\end{proof}

\begin{theorem}\label{th:a2}
Consider two vertices $v_1,v_2$ of $\cBG_+(n)$, associated to families
$\cD_1,\cD_2$ respectively, such that $\bigcap\cD_1=\{i\}\neq\bigcap\cD_2=\{j\}$
and suppose $\cD_1\cap\cD_2=\varnothing$. Then $v_1$ and $v_2$ are adjacent iff
there do not exist collections $\cD_3,\cD_4$ distinct from
$\cD_1,\cD_2$ satisfying

\begin{enumerate}
\item $\cD_3\cap\cD_4=\varnothing$,
\item $\cD_3\cup\cD_4=\cD_1\cup\cD_2$,
\item $\bigcap\cD_3\neq\varnothing$ and $\bigcap\cD_4\neq\varnothing$.
\end{enumerate}
\end{theorem}

\begin{proof}
  $\Rightarrow)$ Suppose there exist
  $\cD_3,\cD_4$ satisfying the above conditions. Then, since $\bigcap\cD_3$ and
  $\bigcap\cD_4$ are nonempty by condition 3), they determine vertices $v_3,v_4$,
  respectively. By conditions 1) and 2), it follows that
  \[
\frac{1}{2}v_1+\frac{1}{2}v_2=\frac{1}{2}v_3+\frac{1}{2}v_4,
\]
showing that $v_1,v_2$ are not adjacent.

$\Leftarrow)$ Suppose $v_1,v_2$ are not adjacent. By Lemma~\ref{lem:a1}, it
follows that there exist vertices $v_3,v_4$ such that $v_1+v_2=v_3+v_4$,
determined by collections $\cD_3,\cD_4$ respectively, satisfying conditions 1)
and 2). Now, condition 3) is implied by the fact that $v_3,v_4$ are vertices.
\end{proof}

\begin{lemma}\label{lem:a3}
  Consider two vertices $v_1,v_2$ of $\cBG_+(n)$, associated to $\cD_1,\cD_2$
  respectively, such that $\bigcap\cD_1=\{i\}\neq\bigcap\cD_2=\{j\}$ and
  $\cD_1\cap\cD_2=\varnothing$. If $v_1,v_2$ are adjacent, then
  \[
\forall S\in \cD_1, j\not\in S, \quad \forall T\in \cD_2, i\not\in T.
  \]
\end{lemma}
\begin{proof}
It suffices to note that if there exists $S\in\cD_1$ such that $j\in S$, then we
can consider $\cD_3:=\cD_1\setminus \{S\}$ and $\cD_4:=\cD_2\cup\{S\}$, which
fulfill the three conditions of Theorem~\ref{th:a2}.
\end{proof}

\begin{theorem}\label{th:a3}
Consider two vertices $v_1,v_2$ of $\cBG_+(n)$, associated to $\cD_1,\cD_2$
respectively, and $\bigcap\cD_1=\{i\}\neq\bigcap\cD_2=\{j\}$. Suppose $\cD_1 \cap \cD_2 \ne \varnothing $ and denote
$\bigcap(\cD_1\cap\cD_2)=T\cup\{i,j\}$, with $T\subseteq N\setminus\{i,j\}$.
Then, $v_1$ and $v_2$ are
adjacent iff the following two conditions are satisfied:

\begin{enumerate}
\item For all $S\in
  \cD_1\setminus \cD_2$, $j\not\in S$, and for all $S\in \cD_2\setminus \cD_1$,
  $i\not\in S$
\item For every disjoint $K_1,K_2\subseteq T$ and disjoint $K_3,K_4\subseteq T$
  such that $K_1\cap K_3\neq\varnothing$ and $K_2\cap K_4\neq\varnothing$, there
  exists either
  $S\in\cD_1\setminus\cD_2$ s.t. $K_1\not\subseteq S$ and $K_2\not\subseteq S$, or
  $S'\in\cD_2\setminus\cD_1$ s.t. $K_3\not\subseteq S'$ and $K_4\not\subseteq S'$.
\end{enumerate}
\end{theorem}

\begin{proof}
  $\Rightarrow)$ Let us consider two adjacent vertices $v_1,v_2$ as above, and
  let $v=\frac{1}{2}(v_1+v_2)$. By Lemma~\ref{lem:a2}, we know that
  $\cD_1\setminus\cD_2\neq\varnothing$ and $\cD_2\setminus
  \cD_1\neq\varnothing$. Suppose there exists $S_1\in\cD_1\setminus \cD_2$
  s.t. $j\in S_1$. Consider $v_3,v_4$ generated by $\cD_1\setminus\{S_1\}$ and
  $\cD_2\cup\{S_1\}$, respectively. Then $v=\frac{1}{2}(v_3+v_4)$, a
  contradiction. The argument is the same with the existence of
  $S_2\in\cD_2\setminus \cD_1$ s.t. $i\in S_2$. This proves the first condition.

  We prove the second condition. Our strategy is to show that the nonexistence
  of a partition of $\cD_1\cup\cD_2$ like in Fig.~\ref{fig:na} (which is
  equivalent to adjacency) implies the second condition. We first observe
  that the first condition implies that a partition of $\cD_1\cup\cD_2$ like in
  cases (b) and (c) of Fig.~\ref{fig:na} can never occur. Indeed, in these
  cases, by the first condition, it follows that $\bigcap\cD_3$ contains neither $i$ nor $j$, and since
  $\bigcap\cD_3\neq\varnothing,$ it must contain some other element, say $k$. But then,
  $k\in\bigcap\cD_1,$ a contradiction. Consequently, we only have to consider case
  (a). Let us call (a)-partition a possible partition like in Case (a) and show that it is not possible to build such a partition. We distinguish
  different cases in terms of $\cD_1\cap\cD_2$.

\begin{enumerate}
\item  Suppose $\bigcap(\cD_1\cap\cD_2)=\{i,j\}$. The first condition implies that
  $j\not\in\bigcap\cB_1$ (same for $\bigcap\cB_2$), and $i\not\in\bigcap\cB_3$
  (same for $\bigcap\cB_4$). It follows that $\bigcap\cD_3=\varnothing$. Hence no
  (a)-partition can exist.

\item Suppose $\bigcap(\cD_1\cap\cD_2)=\{i,j,k\}$. In this case, by the first condition, it follows that $k\in\bigcap\cB_1$ and $k\in\bigcap\cB_3$ to ensure
  $\bigcap\cD_3\neq\varnothing$, and the same holds for $\bigcap\cB_2$ and
  $\bigcap\cB_4$. But then $\bigcap\cD_1\ni k$, a contradiction. Hence no
  (a)-partition can exist.

\item Suppose $\bigcap(\cD_1\cap\cD_2)=T\cup\{i,j\}$ with $|T|\geqslant 2$. There
  exists an (a)-partition iff one can have $(\bigcap\cB_1\cap\bigcap\cB_3)\cap
  T\neq\varnothing$, $(\bigcap\cB_2\cap\bigcap\cB_4)\cap
  T\neq\varnothing$ (to ensure nonemptiness of $\bigcap\cD_3,\bigcap\cD_4$), and
  $\bigcap\cB_1\cap\bigcap\cB_2\cap T=\varnothing$,
  $\bigcap\cB_3\cap\bigcap\cB_4\cap T=\varnothing$ (to ensure
  $\bigcap\cD_1=\{i\}$ and $\bigcap\cD_2=\{j\}$). By letting $K_i:=\bigcap
  \cB_i\cap T$ for $i=1,\ldots,4$, this is equivalent to:
  $\exists K_1,K_2,K_3,K_4\subseteq T$, $K_1\cap K_2=\varnothing$, $K_3\cap
  K_4=\varnothing$, $K_1\cap K_3\neq\varnothing$, $K_2\cap K_4\neq\varnothing$
  such that for every $S\in \cD_1\setminus\cD_2$, either $K_1\subseteq S$ or
  $K_2\subseteq S$, and for $S'\in \cD_2\setminus\cD_1$, either $K_3\subseteq S'$ or
  $K_4\subseteq S'$. Therefore, there is no (a)-partition iff: $\forall
  K_1,K_2,K_3,K_4\subseteq T$ such that  $K_1\cap K_2=\varnothing$, $K_3\cap
  K_4=\varnothing$, $K_1\cap K_3\neq\varnothing$, $K_2\cap K_4\neq\varnothing$,
  either $\exists S\in\cD_1\setminus\cD_2$ s.t. $K_1\not\subseteq
  S,K_2\not\subseteq S$, or $\exists S'\in\cD_2\setminus\cD_1$ s.t. $K_3\not\subseteq
  S',K_4\not\subseteq S$.
\end{enumerate}

  \medskip

  $\Leftarrow)$ Suppose $v_1,v_2 $ are not adjacent. Then by Lemma~\ref{lem:a1},
  there exist $\cD_3,\cD_4$ such that either an (a)-partition, a
  (b)-partition or a (c)-partition of $\cD_1\cup\cD_2$ exists (see Figure~\ref{fig:na}).

  Suppose there exists an (a)-partition. If the first condition is not satisfied,
  we are done. Hence, assume that the first condition holds. Then, proceeding as in cases
  1, 2, 3 in the $\Rightarrow)$ part, we deduce that if $|T|<2$, no
  (a)-partition can exist, and if $|T|\geqslant 2$, the existence of a
  (a)-partition implies the negation of the second condition, as desired.

  Suppose there exists a (b)-partition. Then $\cD_3:=\cD_1\cup\cB_3$ satisfies
  $\bigcap\cD_3\neq\varnothing$, say it contains $k$. Then, $k\in\bigcap\cD_1$,
  which implies $k=i$. This in turn implies $i\in\bigcap\cB_3$, which violates the first
  condition.

  Suppose there exists a (c)-partition.  Observe this cannot occur. Indeed,
  $\bigcap \cD_3 = \bigcap(\cD_1\cup\cD_2)\neq\varnothing$, say $k\in
  \bigcap(\cD_1\cup\cD_2)$. In particular, $k\in\bigcap\cD_1$, which implies
  $k=i$. But then $i\in\bigcap\cD_2$, a contradiction.
\end{proof}

Observe that if $T=\varnothing$ in the above theorem, then only the first condition
remains.

\begin{figure}[hp]
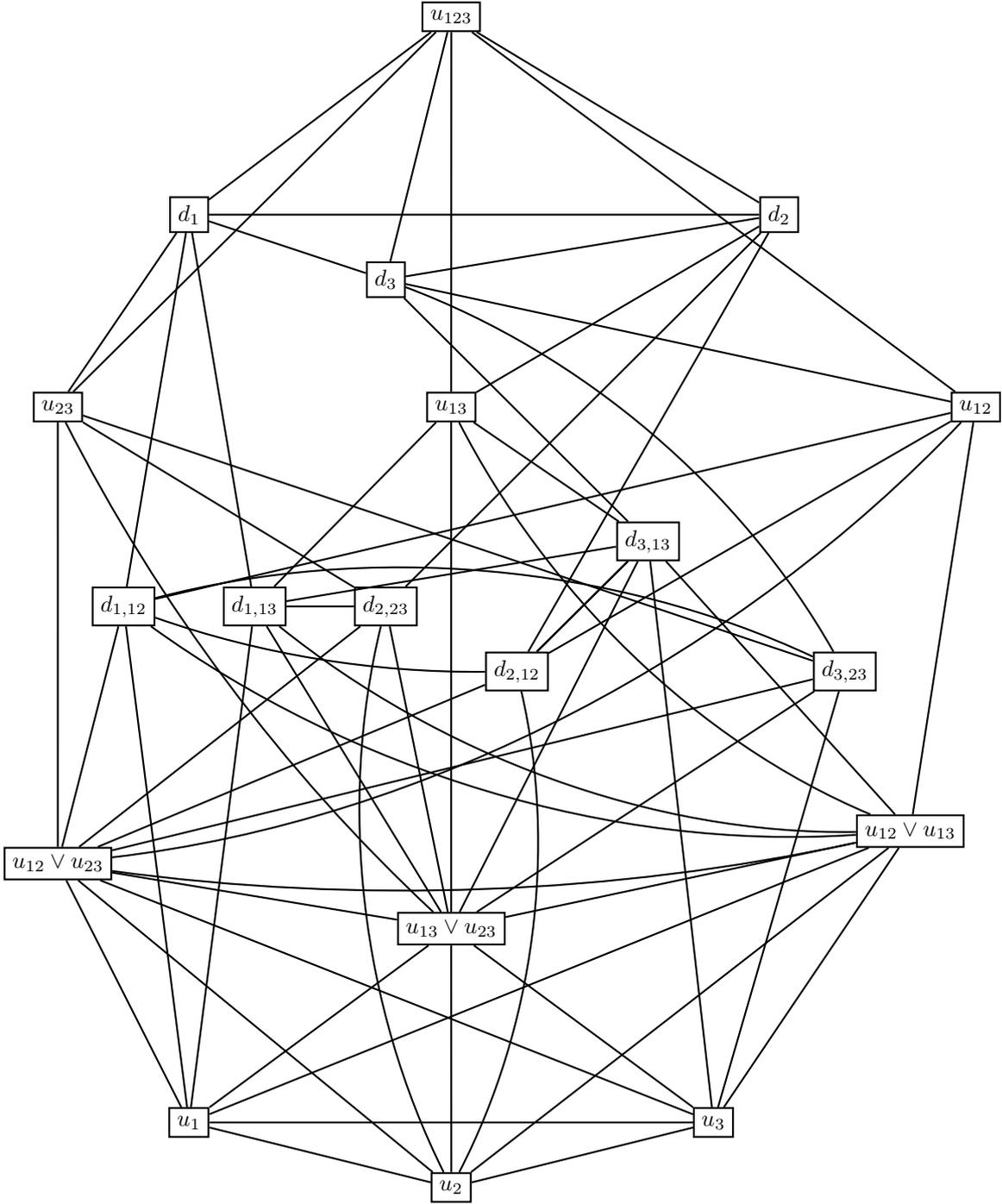

  \begin{center}
   \psset{unit=1cm}
    \pspicture(-0.5,-0.5)(14.5,19.5)
\rput[t](2,1){\rnode{1}{\psframebox{\small $u_1$}}}
\rput[t](6,0){\rnode{2}{\psframebox{\small $u_2$}}}
\rput[t](10,1){\rnode{3}{\psframebox{\small $u_3$}}}
\rput[t](13,5.5){\rnode{12,13}{\psframebox{\small $u_{12}\vee u_{13}$}}}
\rput[t](0,5){\rnode{12,23}{\psframebox{\small $u_{12}\vee u_{23}$}}}
\rput[t](6,4){\rnode{13,23}{\psframebox{\small $u_{13}\vee u_{23}$}}}
\rput[t](1,9){\rnode{1,12}{\psframebox{\small $d_{1,12}$}}}
\rput[t](7,8){\rnode{2,12}{\psframebox{\small $d_{2,12}$}}}
\rput[t](3,9){\rnode{1,13}{\psframebox{\small $d_{1,13}$}}}
\rput[t](9,10){\rnode{3,13}{\psframebox{\small $d_{3,13}$}}}
\rput[t](5,9){\rnode{2,23}{\psframebox{\small $d_{2,23}$}}}
\rput[t](12,8){\rnode{3,23}{\psframebox{\small $d_{3,23}$}}}
\rput[t](14,12){\rnode{12}{\psframebox{\small $u_{12}$}}}
\rput[t](6,12){\rnode{13}{\psframebox{\small $u_{13}$}}}
\rput[t](0,12){\rnode{23}{\psframebox{\small $u_{23}$}}}
\rput[t](2,15){\rnode{d1}{\psframebox{\small $d_{1}$}}}
\rput[t](5,14){\rnode{d3}{\psframebox{\small $d_{3}$}}}
\rput[t](11,15){\rnode{d2}{\psframebox{\small $d_{2}$}}}
\rput[t](6,18){\rnode{123}{\psframebox{\small $u_{123}$}}}

\ncline{-}{1}{12,13}
\ncline{-}{1}{12,23}
\ncline{-}{1}{13,23}
\ncline{-}{1}{2}
\ncline{-}{1}{3}
\ncline{-}{1}{1,12}
\ncline{-}{1}{1,13}
\ncarc[arcangle=-20]{-}{2}{2,12}
\ncline{-}{2}{12,23}
\ncline{-}{2}{12,13}
\ncline{-}{2}{13,23}
\ncline{-}{2}{3}
\ncarc[arcangle=20]{-}{2}{2,23}
\ncline{-}{3}{3,13}
\ncline{-}{3}{13,23}
\ncline{-}{3}{12,23}
\ncline{-}{3}{12,13}
\ncline{-}{3}{3,23}
\ncline{-}{12,13}{12}
\ncarc[arcangle=20]{-}{12,13}{13}
\ncline{-}{12,13}{13,23}
\ncarc[arcangle=10]{-}{12,13}{12,23}
\ncarc[arcangle=20]{-}{12,13}{1,13}
\ncline{-}{12,13}{3,13}
\ncarc[arcangle=20]{-}{12,13}{1,12}
\ncline{-}{12,13}{2.12}
\ncline{-}{13,23}{13}
\ncarc[arcangle=10]{-}{13,23}{23}
\ncline{-}{13,23}{12,23}
\ncline{-}{13,23}{2,23}
\ncline{-}{13,23}{3,23}
\ncline{-}{13,23}{1,13}
\ncline{-}{13,23}{3,13}
\ncarc[arcangle=-20]{-}{12,23}{12}
\ncline{-}{12,23}{23}
\ncline{-}{12,23}{2,23}
\ncline{-}{12,23}{3,23}
\ncline{-}{12,23}{1,12}
\ncline{-}{12,23}{2,12}
\ncline{-}{1,12}{12}
\ncline{-}{1,12}{d1}
\ncarc[arcangle=-10]{-}{1,12}{2,12}
\ncarc[arcangle=20]{-}{1,12}{3,23}
\ncline{-}{2,12}{12}
\ncline{-}{2,12}{d2}
\ncline{-}{2,12}{3,13}
\ncline{-}{1,13}{13}
\ncline{-}{1,13}{d1}
\ncline{-}{1,13}{3,13}
\ncline{-}{1,13}{2,23}
\ncline{-}{3,13}{13}
\ncline{-}{3,13}{d3}
\ncline{-}{3,13}{2,12}
\ncline{-}{2,23}{23}
\ncline{-}{2,23}{d2}
\ncline{-}{3,23}{23}
\ncarc[arcangle=-20]{-}{3,23}{d3}
\ncline{-}{12}{123}
\ncline{-}{12}{d3}
\ncline{-}{d1}{123}
\ncline{-}{d1}{d2}
\ncline{-}{d1}{d3}
\ncline{-}{13}{123}
\ncline{-}{13}{d2}
\ncline{-}{d2}{123}
\ncline{-}{d2}{d3}
\ncline{-}{23}{123}
\ncline{-}{23}{d1}
\ncline{-}{d3}{123}
\endpspicture
  \end{center}
  \caption{Adjacency graph of $\cBG_+(n)$ for $n=3.$ Vertices are defined in Table~\ref{vert_n=3}}
  \label{AdjacencyGraph}
\end{figure}

In Figure \ref{AdjacencyGraph} the adjacency graph of vertices of $\cBG_+(n)$ can be seen. This graph has the vertices as nodes and there is an
edge between two nodes if they are adjacent. Finally, to illustrate Theorem \ref{path}, in Figure \ref{hamilton}, a Hamilton path joining $u_{123}$ and $d_{2,23}$ is given.
\begin{figure}[hp]
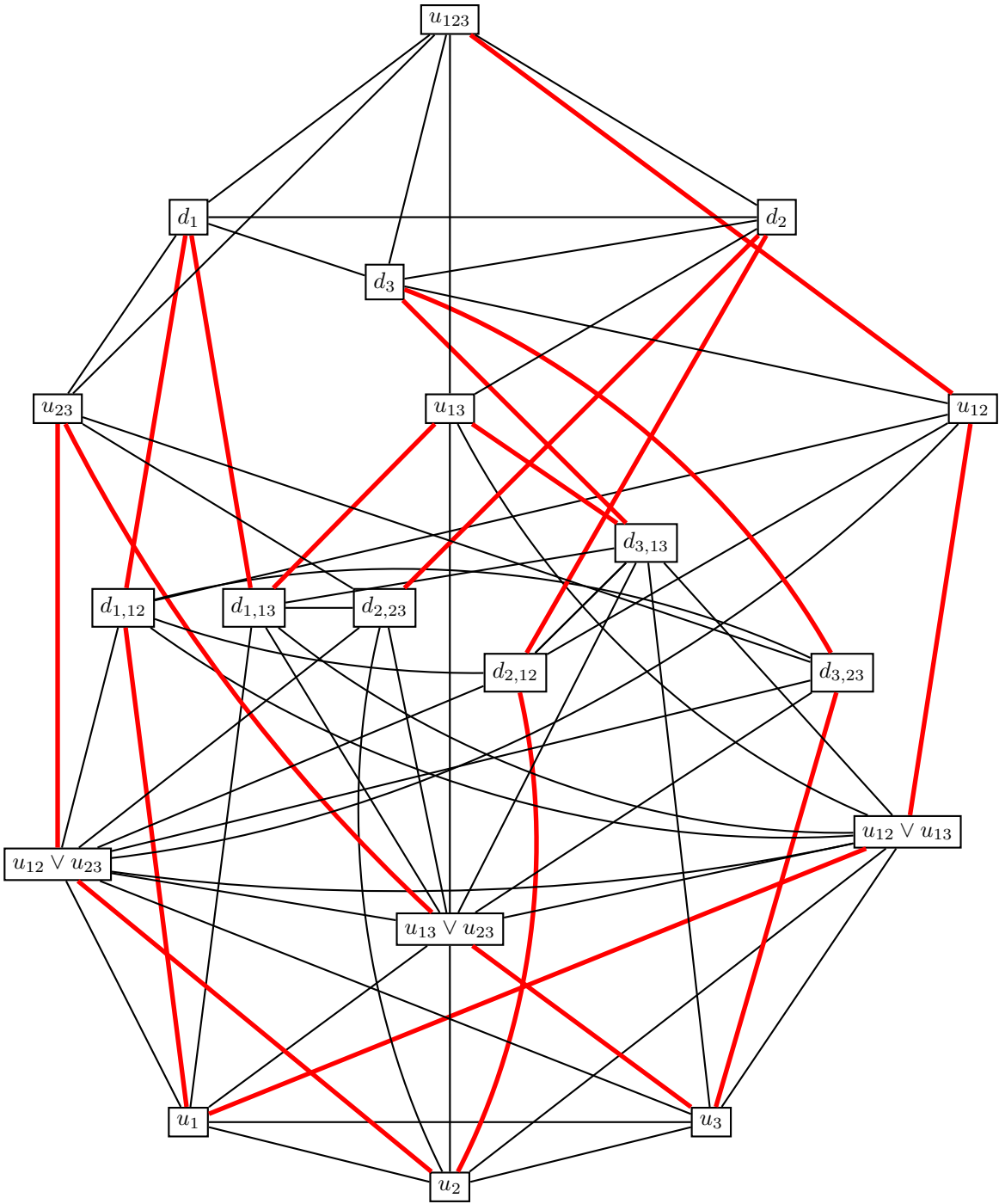

  \begin{center}
     \psset{unit=1cm}
    \pspicture(-0.5,-0.5)(14.5,19.5)
\rput[t](2,1){\rnode{1}{\psframebox{\small $u_1$}}}
\rput[t](6,0){\rnode{2}{\psframebox{\small $u_2$}}}
\rput[t](10,1){\rnode{3}{\psframebox{\small $u_3$}}}
\rput[t](13,5.5){\rnode{12,13}{\psframebox{\small $u_{12}\vee u_{13}$}}}
\rput[t](0,5){\rnode{12,23}{\psframebox{\small $u_{12}\vee u_{23}$}}}
\rput[t](6,4){\rnode{13,23}{\psframebox{\small $u_{13}\vee u_{23}$}}}
\rput[t](1,9){\rnode{1,12}{\psframebox{\small $d_{1,12}$}}}
\rput[t](7,8){\rnode{2,12}{\psframebox{\small $d_{2,12}$}}}
\rput[t](3,9){\rnode{1,13}{\psframebox{\small $d_{1,13}$}}}
\rput[t](9,10){\rnode{3,13}{\psframebox{\small $d_{3,13}$}}}
\rput[t](5,9){\rnode{2,23}{\psframebox{\small $d_{2,23}$}}}
\rput[t](12,8){\rnode{3,23}{\psframebox{\small $d_{3,23}$}}}
\rput[t](14,12){\rnode{12}{\psframebox{\small $u_{12}$}}}
\rput[t](6,12){\rnode{13}{\psframebox{\small $u_{13}$}}}
\rput[t](0,12){\rnode{23}{\psframebox{\small $u_{23}$}}}
\rput[t](2,15){\rnode{d1}{\psframebox{\small $d_{1}$}}}
\rput[t](5,14){\rnode{d3}{\psframebox{\small $d_{3}$}}}
\rput[t](11,15){\rnode{d2}{\psframebox{\small $d_{2}$}}}
\rput[t](6,18){\rnode{123}{\psframebox{\small $u_{123}$}}}

\ncline[linecolor=red,linewidth=2pt]{-}{1}{12,13}
\ncline{-}{1}{12,23}
\ncline{-}{1}{13,23}
\ncline{-}{1}{2}
\ncline{-}{1}{3}
\ncline[linecolor=red,linewidth=2pt]{-}{1}{1,12}
\ncline{-}{1}{1,13}
\ncarc[arcangle=-20,linecolor=red,linewidth=2pt]{-}{2}{2,12}
\ncline[linecolor=red,linewidth=2pt]{-}{2}{12,23}
\ncline{-}{2}{12,13}
\ncline{-}{2}{13,23}
\ncline{-}{2}{3}
\ncarc[arcangle=20]{-}{2}{2,23}
\ncline{-}{3}{3,13}
\ncline[linecolor=red,linewidth=2pt]{-}{3}{13,23}
\ncline{-}{3}{12,23}
\ncline{-}{3}{12,13}
\ncline[linecolor=red,linewidth=2pt]{-}{3}{3,23}
\ncline[linecolor=red,linewidth=2pt]{-}{12,13}{12}
\ncarc[arcangle=20]{-}{12,13}{13}
\ncline{-}{12,13}{13,23}
\ncarc[arcangle=10]{-}{12,13}{12,23}
\ncarc[arcangle=20]{-}{12,13}{1,13}
\ncline{-}{12,13}{3,13}
\ncarc[arcangle=20]{-}{12,13}{1,12}
\ncline{-}{12,13}{2.12}
\ncline{-}{13,23}{13}
\ncarc[arcangle=10,linecolor=red,linewidth=2pt]{-}{13,23}{23}
\ncline{-}{13,23}{12,23}
\ncline{-}{13,23}{2,23}
\ncline{-}{13,23}{3,23}
\ncline{-}{13,23}{1,13}
\ncline{-}{13,23}{3,13}
\ncarc[arcangle=-20]{-}{12,23}{12}
\ncline[linecolor=red,linewidth=2pt]{-}{12,23}{23}
\ncline{-}{12,23}{2,23}
\ncline{-}{12,23}{3,23}
\ncline{-}{12,23}{1,12}
\ncline{-}{12,23}{2,12}
\ncline{-}{1,12}{12}
\ncline[linecolor=red,linewidth=2pt]{-}{1,12}{d1}
\ncarc[arcangle=-10]{-}{1,12}{2,12}
\ncarc[arcangle=20]{-}{1,12}{3,23}
\ncline{-}{2,12}{12}
\ncline[linecolor=red,linewidth=2pt]{-}{2,12}{d2}
\ncline{-}{2,12}{3,13}
\ncline[linecolor=red,linewidth=2pt]{-}{1,13}{13}
\ncline[linecolor=red,linewidth=2pt]{-}{1,13}{d1}
\ncline{-}{1,13}{3,13}
\ncline{-}{1,13}{2,23}
\ncline[linecolor=red,linewidth=2pt]{-}{3,13}{13}
\ncline[linecolor=red,linewidth=2pt]{-}{3,13}{d3}
\ncline{-}{3,13}{2,12}
\ncline{-}{2,23}{23}
\ncline[linecolor=red,linewidth=2pt]{-}{2,23}{d2}
\ncline{-}{3,23}{23}
\ncarc[arcangle=-20,linecolor=red,linewidth=2pt]{-}{3,23}{d3}
\ncline[linecolor=red,linewidth=2pt]{-}{12}{123}
\ncline{-}{12}{d3}
\ncline{-}{d1}{123}
\ncline{-}{d1}{d2}
\ncline{-}{d1}{d3}
\ncline{-}{13}{123}
\ncline{-}{13}{d2}
\ncline{-}{d2}{123}
\ncline{-}{d2}{d3}
\ncline{-}{23}{123}
\ncline{-}{23}{d1}
\ncline{-}{d3}{123}
\endpspicture
  \end{center}
  \caption{Adjacency graph of $\cBG_+(n)$ for $n=3.$ The Hamiltonian path
    joining $u_{123}$ and $d_{2,23}$ is indicated in thick red line.}
  \label{hamilton}
\end{figure}

\subsection{Facets}
 The following result gives the facets of $\cBG_+(n).$

\begin{theorem}\label{th:facet}
The following holds:
\begin{enumerate}
\item Each equality $v(S)=0$, $S\neq\varnothing,N$, defines  a facet.
\item  Each equality $\sum_{S\in\cB}\lambda^\cB_Sv(S)=1$, $\cB\in\gB^*(n)$, defines
  a facet.
\item  The number of facets is $2^n + b(n)-3,$ with $b(n)$ the number of minimal
  balanced collections on $N$.
\end{enumerate}
\end{theorem}

\begin{proof}
Recall from (\ref{eq:bg+}) that
$$ \cBG_{+}(n):=\Big\lbrace v \in \cG(n): \sum_{S\in \cB} \lambda_S^{\cB} v(S)
\leqslant 1 , \forall \cB \in \gB^*(n), v(S)\geqslant 0,\forall \varnothing\neq S\subset N, v(N)=1\Big\rbrace .$$
All the facets of a polytope can be obtained by converting into equality one of
the inequalities that define it. Note however that some inequalities might
define a lower dimensional face.

1. We claim that every equality $v(S)=0, S\ne N, \varnothing $, leads by
intersection with $\cBG_+(n)$ to a face of dimension $2^n-3$, hence a facet. To see this,
consider the games $\hat{\delta}_T=\delta_T+\delta_N, T\ne S, N, \varnothing $,
i.e.,
\[
 \hat{\delta}_T(S')=\begin{cases}
 1, & \text{ if } S'=T \text{ or } S'=N\\
 0, & \text{ otherwise}.
 \end{cases}
 \]
Clearly, the games $\hat{\delta}_T$ for every $T\neq\varnothing, S,N$ are
vertices of
$\cBG_{+}(n)$ and satisfy $v(S)=0$. In addition, the game $u_N$ also belongs to
the face $v(S)=0$. These $2^n-2$ games are affinely independent because the $2^n-3$
games $\hat{\delta}_T-u_N=\delta_T$ are linearly independent, which proves the claim.

2. Consider any m.b.c. $\cB\in\gB^*(n)$ and the corresponding inequality. For any
$i\in N$ define the game $v^i$ by
\[
v^i(S)=\begin{cases}
1, & \text{ if } S\ni i \text{ and } S\in\cB, \text{ or } S=N\\
0, & \text{otherwise}.
\end{cases}
\]
Observe that each $v^i$ is a vertex of $\cBG_+(n)$, and $v^i$
satisfies $\sum_{S\in\cB}\lambda^\cB_Sv^{ i}(S)=1$. Moreover, the number of distinct
$v^i$ is $|\cB|=:p$. Let us call $\cD^1,\ldots,\cD^n$ the corresponding
collections defining $v^1,\ldots,v^n$. Define successively the collections
\begin{itemize}
\item $\cD^1\cup \{S\}$, with $S\ni 1$, $S\not\in \cB$
\item $\cD^2\cup\{S\}$, with $S\ni 2$, $S\not\ni 1$, $S\not\in \cB$
\item $\cD^3\cup\{S\}$, with $S\ni 3$, $S\not\ni 2$, $S\not\ni 1$, $S\not\in\cB$
\item etc.
\item $\cD^n\cup\{\{n\}\}$ (if $\{n\}$ not already present in $\cD^n$)
\end{itemize}
Observe that $\cD^1,\ldots,\cD^n$ plus all the above ones yields $2^n-2$
distinct collections, where each $S\in2^N\setminus\{\varnothing,N\}$ is present
at least once. Moreover, each collection defines a game which is a vertex
belonging to the face defined by $\sum_{S\in\cB}\lambda^\cB_Sv(S)=1$. It follows
that these $2^n-2$ games are affinely independent (as, e.g., the family of the
$2^n-3$ games $v-v^1$ with $v$ any game as defined above forms a linearly
independent family), and therefore the inequality
$\sum_{S\in\cB}\lambda^\cB_Sv(S)\leqslant 1$ defines a facet.

3. This is an immediate consequence of 1) and 2).
\end{proof}

\section{When is the core reduced to a singleton?}\label{sec:pc}
We address in this section the question of which balanced game in any of the
three sets of balanced games under consideration in this paper has a core
reduced to a singleton. Let us call this for simplicity a  {\it point core}.

A first observation is the following: Since $\cBG_\alpha(n)$ and $\cBG_+(n)$ are
subsets of $\cBG(n)$, it suffices to find all games in $\cBG(n)$
which have a point core and then to check if they belong to $\cG_+(n)$ or
satisfy $v(N)=\alpha$. Therefore we put our effort into finding all balanced
games in $\cBG(n)$ having a point core.

\subsection{The case of $\cBG(n)$}

For a given balanced game $v$ in $\cBG(n)$, let us denote by $\cE(v)$ the collection
\[
\cE(v) = \{S\subseteq N\mymid x(S)=v(S), \forall x\in C(v)\}.
\]
The set $\cE(v)$ is called the set of {\it effective} coalitions for $v.$
Obviously, $N\in\cE(v)$, and $\cE(v)=\{N\}$ if and only if the core is
full-dimensional, i.e., $(n-1)$-dimensional.
The following result of Laplace Mermoud et al. \cite{lagrsu23}
is central in our investigation.
\begin{lemma}\label{lem:lmgs}
  $\cE(v)$ is the union of all minimal balanced collection $\cB$ such that
  \[
\sum_{S\in\cB}\lambda_S^\cB v(S) = v(N).
  \]
\end{lemma}

An immediate consequence of this lemma is the following general result.
\begin{proposition}\label{prop:int}
If $v$ belongs to the interior of $\cBG(n),$ then $v$ has not a point core.
\end{proposition}
Indeed, for such a game $v$, no inequality in (\ref{eq:bg}) is tight, which
implies by Lemma~\ref{lem:lmgs} that $\cE(v)$ is reduced to $N$, which means that
the core is full-dimensional.

We begin with a simple result.
\begin{lemma}\label{lem:lin}
Any game in $\Lin(\cBG(n))$ has a point core.
\end{lemma}

\begin{proof}
Take $v\in\Lin(\cBG(n))$. Then, this game can be written as

$$ v=\sum_{i\in N}\alpha_iw_i, \, \alpha_1,\ldots,\alpha_n\in\RR .$$

As each $w_i=u_{\{i\}}$ is an additive game, so is $v$, and therefore $v$ has a point core.
\end{proof}

We turn to the examination of facets, which by Theorem~\ref{th:bgfacet} correspond to
minimal balanced collections in $\gB^*(n)$.
\begin{theorem}\label{th:pcbg}
Consider a minimal balanced collection $\cB\in\gB^*(n)$ and its corresponding
facet $\cF$ in $\cBG(n)$. The following holds:

\begin{enumerate}
\item If $|\cB|=n$, every game in $\cF$ has a point core.
\item Otherwise, no game in the relative interior of $\cF$ has a point core.
\end{enumerate}
\end{theorem}

\begin{proof}
Consider a minimal balanced collection $\cB$ and its corresponding facet $\cF$
in $\cBG(n)$.

\begin{enumerate}
\item Take $v$ in the relative interior of $\cF$. Note that any core
element $x\in C(v)$ satisfies the system

\begin{equation}\label{eq:s}
x(S) = v(S), \quad S\in\cE(v).
\end{equation}
By Lemma~\ref{lem:lmgs}, we have that $\cE(v)=\cB \cup \{ N\} .$  Observe
  that the inequality $x(N)=v(N)$ is redundant with the others as $\cB$ is
  balanced. Therefore, in (\ref{eq:s}) we can remove $S=N$.
Consider now the system

\[
\sum_{S\in\cB}\lambda_S1^S=1^N.
\]
As $\cB$ is minimal, the above system has a unique solution. Hence, the rank of
the matrix $M$ of this linear system is $|\cB|$. On the other hand, the matrix
of system (\ref{eq:s}) (without $N$) is the transpose of $M$, and
therefore has rank $|\cB|$. It follows that the solution of (\ref{eq:s}) is
unique iff $|\cB|=n$.

\item It remains to prove that games on the frontier of a facet $\cF$ defined by
  a m.b.c. $\cB$ s.t. $|\cB|=n$ have also a point core. Take $v$ in the frontier
  of $\cF$. Then, $v$ also belongs to other facets, say $\cF_1, ..., \cF_r$,
  associated with m.b.c. $\cB_1, ..., \cB_r,$ respectively. It follows from
  Lemma~\ref{lem:lmgs} that any core element $x\in C(v)$ satisfies the system

\[
x(S) = v(S), \quad S\in\cB\cup\cB_1 \cup ... \cup \cB_r.
\]

As the system $x(S)=v(S)$ for $S\in \cB$ has a unique solution, it follows that
either the above system has no solution or a unique one. But the system has solution by the fact that $v$ is balanced.
\end{enumerate}
\end{proof}

\begin{remark}
As the lineality space is the intersection of all facets,
Lemma~\ref{lem:lin} is obtained as a corollary of the above theorem.
\end{remark}

Note that Part 2 of the previous proof allows for a characterization of games with a point core in $\cBG(n).$

\begin{theorem}
Consider a face $\cF = \cF_1 \cap \cdots \cap \cF_p,$ with $\cF_1, \ldots, \cF_p$ the facets associated to m.b.c. $\cB_1,\ldots,\cB_p,$ respectively. Then, any game in $\cF$ has a point
core iff the rank of the matrix $\{1^S,S\in \cB_1\cup\cdots\cup\cB_p\}$ is $n$.
\end{theorem}

\begin{proof}
Using the same argument as in the proof of Theorem~\ref{th:pcbg}, any core
element $x$ of a game $v$ in the relative interior of $\cF$ satisfies the system

\[
x(S) = v(S), \quad S\in\cB_1\cup\cdots\cup\cB_p
\]
and the conclusion follows immediately. Now, if $v$ is in the frontier, proceed
as in the proof of  Theorem~\ref{th:pcbg}.
\end{proof}

The above results plus the result on the interior of $\cBG(n)$ completely
characterizes the set of games with a point core.

\begin{example}
Consider $n=4$, and the facets $\cF_1, \cF_2$ associated with the
m.b.c. $\cB_1=\{12,34\}$ and $\cB_2=\{1,234\}$. Then, games in the relative
interior of $\cF_1\cap \cF_2$ have
no point core.  Although $|\cB_1|+|\cB_2|=4=n$, the rank of the matrix $\{1^S,S\in \cB_1\cup \cB_2\}$ is 3,
therefore games in the face $\cF_1\cap \cF_2$ have no point core but a core of
dimension 1.

Consider now the face $\cF_3$ associated with $\cB_3=\{13,124,23\}$. Games in
the relative interior of this facet have no point core. However, games in
$\cF'=\cF_1\cap \cF_3$ have a point core since the rank of $\{\1^S,S\in \cB_1\cup \cB_3\}$ is 4.
\end{example}

\begin{example}
Let us give a complete analysis with $n=3$. The lineality space has basis
  $\{u_{\{1\}},u_{\{2\}},u_{\{3\}}\}$. The extremal rays are
  $-\delta_{12},-\delta_{13},-\delta_{23}$, and

\begin{align*}
    r_1 &= \delta_{12} + \delta_{13} + \delta_{123}\\
    r_2 &= \delta_{12} + \delta_{23} + \delta_{123}\\
    r_3 &= \delta_{13} + \delta_{23} + \delta_{123}.
\end{align*}

Next table gives the facets and which extremal rays (except those of the lineality space which all belong to every facet) belong to them.

\begin{center}
    \begin{tabular}{|c|c|c|c|c|c|c|}\hline
        m.b.c. & $-\delta_{12}$ & $-\delta_{13}$ & $-\delta_{23}$ & $r_1$ &
        $r_2$ & $r_3$ \\ \hline
        $\cB_1=\{1,2,3\}$ & $\times$ & $\times$ & $\times$ & & & \\
        $\cB_2=\{1,23\}$ &  $\times$ & $\times$ & & &  $\times$ & $\times$ \\
        $\cB_3=\{2,13\}$ &   $\times$ & & $\times$ &  $\times$ & & $\times$ \\
        $\cB_4=\{3,12\}$ &   & $\times$ & $\times$ &  $\times$ & $\times$ & \\
        $\cB_5=\{12,13,23\}$ & & & & $\times$ & $\times$ & $\times$ \\ \hline
    \end{tabular}
  \end{center}

  This shows the neighborhood relations between facets (two facets are neighbors
  if they have a common extremal ray which is not in the lineality space), and
  consequently all faces. Interestingly, the figure below of the 5 facets gives a faithful
  representation of the neighborhood relation, although it is not a correct
  geometrical representation. The part in blue indicates where are the games
  with a point core.

  \begin{center}
    \psset{unit=0.7cm}
    
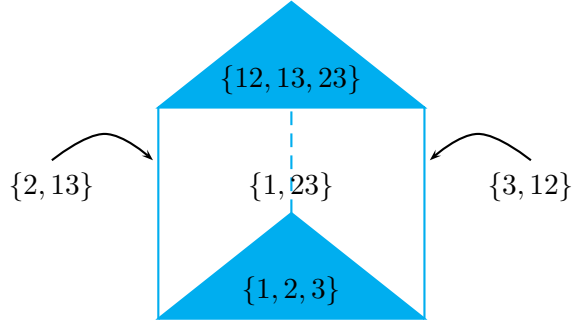
\pspicture(-2,0)(7,6)
    \pspolygon[fillstyle=solid,fillcolor=cyan,linecolor=cyan](0,0)(5,0)(2.5,2)
    \pspolygon[fillstyle=solid,fillcolor=cyan,linecolor=cyan](0,4)(5,4)(2.5,6)
    \psline[linecolor=cyan](0,0)(0,4)
    \psline[linecolor=cyan](5,0)(5,4)
    \psline[linestyle=dashed,linecolor=cyan](2.5,2)(2.5,6)
    \uput[90](2.5,0){$\{1,2,3\}$}
    \uput[90](2.5,4){$\{12,13,23\}$}
    \rput(2.5,2.5){$\{1,23\}$}
    \rput(-2,2.5){$\{2,13\}$}
    \rput(7,2.5){$\{3,12\}$}
    \pscurve{->}(-2,3)(-1,3.5)(-0.1,3)
    \pscurve{->}(7,3)(6,3.5)(5.1,3)
    \endpspicture
  \end{center}
  \end{example}

\subsection{The case of $\cBG_\alpha(n)$}
Since the facet of $\cBG_\alpha(n)$ corresponding to some m.b.c. $\cB\in\gB^*(n)$
is the intersection of the facet of $\cBG(n)$ corresponding to $\cB$ and the
hyperplane $v(N)=\alpha$, every result established for $\cBG(n)$ applies without
any change to $\cBG_\alpha(n)$.

\subsection{The case of $\cBG_+(n)$}
We know by Theorem~\ref{th:facet} that every
m.b.c. $\cB\in\gB^*(n)$ determines  a facet of $\cBG_+(n)$,
which is therefore a subset of the corresponding facet in $\cBG(n)$. Therefore,
once again, Theorem~\ref{th:pcbg} applies without any change.

On the other hand, we know from Theorem~\ref{th:facet} that each equality
$v(S)=0$, $S\neq\varnothing,N$, defines a facet. Taking a game $v$ in the
relative interior of this facet, as no inequality
$\sum_{S\in\cB}\lambda^\cB_Sv(S)\leqslant 1$ is tight, $v$ is in the interior of
$\cBG(n)$, and by Proposition~\ref{prop:int}, $v$ has not a point core.
Summarizing, we have found the following.
\begin{theorem}
  The following holds for $\cBG_+(n)$:
  \begin{enumerate}
    \item Every game in a facet defined by $\cB\in\gB^*(n)$ s.t. $|\cB|=n$ has a
      point core.
    \item No game in the relative interior of any other facet has a point core.
    \item Consider a face $\cF=\cF_1\cap\cdots\cap\cF_p\cap \cF_{p+1}\cap\cdots\cap\cF_r$, with
      $\cF_1,\ldots,\cF_p$ being facets associated to
      m.b.c. $\cB_1,\ldots,\cB_p$, and $\cF_{p+1},\ldots,\cF_r$ being facets
      associated to sets $S_{p+1},\ldots, S_r$. Then any game in $\cF$ has a
      point core iff the rank of the matrix $\{\1^S,
      S\in\cB_1\cup\cdots\cup\cB_p\}$ is $n$.
    \end{enumerate}
\end{theorem}
Note that in (3) of the above theorem, the facets associated to sets
  $S_{p+1},\ldots,S_r$ play no role.

Another approach is to address the question under the point of view of vertices,
instead of faces.  Proposition~\ref{prop:v} gives immediately the answer to our
question as far as vertices are concerned.
\begin{corollary}
A vertex $v\in\cBG_+(n)$ associated to collection $\cD$ has a point core iff
$|\bigcap\cD|=1$.
\end{corollary}

Another point is to investigate whether games in a face whose all vertices have a point
core have also a point core. Surprisingly, this is false in general already for
edges, i.e., faces of dimension 1. The next theorem clarifies the situation for
edges.
\begin{theorem}\label{th:a4}
  Consider two adjacent vertices $v_1,v_2$ of $\cBG_+(n)$, with associated
  collections $\cD_1,\cD_2$ respectively, and $\bigcap\cD_1=\{i\}$,
  $\bigcap\cD_2=\{j\}$. Consider $v=\lambda v_1+(1-\lambda)v_2$ with
  $\lambda\in\left]0,1\right[$, i.e., a game in the edge between $v_1,v_2$. Then:
  \begin{enumerate}
  \item If $i=j$, then $C(v)$ is a singleton, i.e., $v$ has a point core.
  \item If $i\neq j$ and $n\leqslant 4$, then $v$ has a point core.
  \end{enumerate}
\end{theorem}
\begin{proof}
1. Suppose $i=j$. Then by Theorem~\ref{th:a1}, we may suppose $\cD_1\subseteq
\cD_2$. Take $x\in C(v)$. Then for any $S\in\cD_1$, $v(S)=1$, therefore
$x(S)\geqslant 1$, which implies $x\in C(v_1)$. As $v$ is balanced and $C(v_1)$
is a singleton, so is $C(v)$.

2. Suppose w.l.o.g. $i=1,j=2$, and $n\leqslant 4$.

2.1. Suppose $\cD_1\cap\cD_2=\varnothing$ and consider $x\in C(v)$. Suppose
$\{1\}\in\cD_1$. Then $x_1\geqslant \lambda$. Similarly, if $\{2\}\in\cD_2$,
$x_2\geqslant 1-\lambda$. As $x(N)=1$ and $x\geqslant 0$, both facts force
$x_1=\lambda$ and $x_2=1-\lambda$, and therefore $x_k=0$ for all $k\neq 1,2$,
proving that $v$ has a point core. Otherwise, if $\{2\}\not\in\cD_2$, then
$\cD_2$ must contain at least two sets, whose intersection yields $\{2\}$. This
is possible with $n=4$ and then $\cD_2$ must contain $\{2,3\},\{2,4\}$ by Lemma~\ref{lem:a3}. Then, $x_2+x_3\geqslant 1-\lambda$ and
$x_2+x_4\geqslant 1-\lambda$, which together with $x_1\geqslant \lambda$ forces
equality everywhere, from which we deduce $x_3=x_4$, then $x_1+x_2+x_3=1$, and
so $x_4=x_3=0$ and $x_2=1-\lambda$. Finally,
if $\{1\}\not\in\cD_1$ and $\{2\}\not\in\cD_2$, we get
$\cD_1$ contains $\{1,3\},\{1,4\}$ and $\cD_2$ contains $\{2,3\},\{2,4\}$, which contradicts
adjacency (see case (a) of Fig.~\ref{fig:na}).

2.2 Suppose $\cD_1\cap\cD_2\neq\varnothing$ and consider $x\in C(v)$. Any set in
$\cD_1\cap \cD_2$ must contain $\{1,2\}$. Suppose first $\{1,2\} \in\cD_1\cap
\cD_2$. Then $x_1+x_2\geqslant 1$, and since $x(N)=1$ and $x\geqslant 0$, we
obtain $x_1+x_2=1$ and $x_k=0$ for $k\neq 1,2$. Since there must exist
  $S\in\cD_1\setminus\cD_2$ s.t. $2\not\in S$, this implies that
  $x(S)=x_1\geqslant \lambda$. Similarly, there exists $S\in\cD_2\setminus\cD_1$
  such that $1\not\in S$, and $x(S)=x_2\geqslant 1-\lambda$. Combining this
with $x_1+x_2=1$, we obtain equality throughout, which proves that $v$ has a
point core.

Suppose now with $2<n\leqslant 4$ that $\{1,2,3\}\in\cD_1\cap\cD_2$. Proceeding
as above we get $x_1+x_2+x_3=1$ and $x_k=0$ for $k\neq 1,2,3$. Now, there must
exist $S\in\cD_1\setminus\cD_2$ such that $3\not\in S$: either $S=\{1\}$ or, if
$n=4$, $S=\{1,4\}$. This yields $x_1\geqslant \lambda$. Similarly with
$\cD_2\setminus\cD_1$, we get $x_2\geqslant 1-\lambda$, and we can conclude as
above. The case $\{1,2,4\}\in\cD_1\cap\cD_2$ is similar.

Finally, with $n=4$, $\{1,2,3,4\}\in\cD_1\cap \cD_2$ is possible. Then,
$\cD_1\setminus\cD_2$ must contain $\{1\}$ or both $\{1,3\},\{1,4\}$, and
similarly for $\cD_2\setminus\cD_1$. It suffices to proceed as for the case
$\cD_1\cap \cD_2=\varnothing$.
\end{proof}

The following counterexample shows that the result is no more true for $n>4$.
\begin{example}
  Let us take $n=5$ and two collections $\cD_1,\cD_2$ defined by
  \[
\cD_1= \{\{1,2,3,4,5\}, \{1,3\}, \{1,4\}, \{1,5\}\}, \quad \cD_2= \{\{1,2,3,4,5\},
\{2,3,4\}, \{2,4,5\}, \{2,3,5\}\}.
\]
We have $\bigcap\cD_1=\{1\}$ and $\bigcap\cD_2=\{2\}$. Therefore, $\cD_1,\cD_2$
determine vertices, which we denote by $v_1,v_2$, respectively. It can be
checked (via Theorem~\ref{th:a3} or simply Fig.~\ref{fig:na}) that these two vertices are
adjacent.

Let us consider $v=\frac{1}{2}(v_1+v_2)$. One can check that $x,y$ are two
distinct core elements of $v$, hence it is not reduced to a singleton:
\begin{align*}
x &=\begin{pmatrix}\frac{1}{4} & 0 &\frac{1}{4} &\frac{1}{4}
      &\frac{1}{4}\end{pmatrix}\\
y &=\begin{pmatrix}\frac{1}{2} & \frac{1}{2} & 0 & 0 & 0\end{pmatrix}.
\end{align*}
\end{example}

\section{Applications}

The fact that the polytope $\cBG_{+}(n)$ is combinatorial has profound
  implications on optimization issues associated with the polytope. In \cite{mata95}, the importance of a polytope being combinatorial when
  performing optimization on it is explained. In this work, the property of
  being combinatorial is referred to as Property C. Some results of this work
  require a stronger property called Property B (Property B implies Property
  C). Next, we can see that the polytope $\cBG_{+}(n)$ also satisfies Property
  B, and therefore all the results of  \cite{mata95} are applicable to our case.

\begin{proposition}
Let $v_1, v_2$ and $v_3$ three vertices of $\cBG_{+}(n)$ such that $v_1 \leqslant v_2 \leqslant v_3$ (coordinate-wise), then $v_1+v_3-v_2$ is a vertex of $\cBG_{+}(n)$ (Property B).
\end{proposition}

\begin{proof}
Let $\cD_1, \cD_2$ and $\cD_3$ be the collections associated with each vertex. Since $\cD_1\subseteq \cD_2 \subseteq \cD_3,$ it follows that the sets attaining value 1 in $v_1 + v_3 - v_2$ are exactly those in $\cD_3\setminus (\cD_2 \cup \cD_1).$ As this set is contained in $\cD_3,$ it follows that

$$ \bigcap (\cD_3\setminus (\cD_2 \cup \cD_1)) \supseteq \bigcap \cD_3 \ne \emptyset ,$$ and
 therefore $v_1+v_3-v_2 \in \cBG_+(n).$
\end{proof}

For combinatorial polytopes, more efficient vertex enumeration algorithms can be defined. In \cite{memu23}, one can appreciate how the presence of Hamiltonian paths between vertices can be leveraged to make enumeration algorithms more efficient.

\medskip

Many interesting problems in cooperative game theory can be expressed as an optimization problem on $\cBG_+(n)$, especially approximation problems. As
explained above, the fact
that $\cBG_+(n)$ is combinatorial and even satisfies the stronger property B
permits to benefit from more efficient vertex enumeration algorithms and
therefore to obtain better performance in linear optimization.

Many operators on games are linear, e.g., the Harsanyi dividends
(a.k.a. M\"obius transform), the Shapley value and its
generalization the interaction transform $I:v\mapsto I^v$ defined by
\[
I^v(S) = \sum_{T\subseteq N\setminus
  S}\frac{(n-|S|-|T|)!|T|!}{(n-|S|+1)!}\sum_{K\subseteq S}(-1)^{|S\setminus
  K|}v(K\cup T)
\]
(see  \cite{grmaro99a,gra16} for many other linear operators on games). Note
that $I^v(\{i\})$ is the Shapley value for player $i$.
It follows that the maximization/minimization of $I^v(S)$ over $v\in\cBG_+(n)$
for some $S$, or any linear combination of such terms, is a LP problem which can be solved
efficiently, taking advantage of the fact that $\cBG_+(n)$ is combinatorial, and
even satisfies property B.

Another class of optimization problems is the approximation problem. As
explained in the introduction, the main motivation behind this work is to be
able to solve the projection problem on $\cBG(n)$, or on
$\cBG_\alpha(n),\cBG_+(n)$, which is a quadratic optimization problem. One can
also consider to minimize the $L_1$ norm instead of the $L_2$ norm, which leads
to a LP problem (after standard linearization):
\[
\min_{v\in\cBG_+(n)} \sum_{S\subseteq N}|v(S)-w(S)|
\]
with $w\in\cG_+(n)$. However, the linearization implies to introduce new variables and
constraints, therefore it is not obvious if one can still benefit from the
properties of $\cBG_+(n)$. Some further investigation is needed here.

\section{Concluding remarks}\label{sec:con}
Our study has permitted to have a complete description of the polyhedral
structure of the set of balanced games $\cBG(n)$, as well as of its subsets
$\cBG_\alpha(n)$ and $\cBG_+(n)$, in terms of extremal rays, vertices, facets
and adjacency relations between vertices. In addition, we have provided an
algorithm for the uniform random generation of the vertices of
$\cBG_+(n)$. Unexpectedly, the polytope $\cBG_+(n)$ of nonnegative normalized
balanced games seems to be related to some well-known combinatorial structures, as its
number of vertices is a known sequence in OEIS, related to the number of Boolean
functions in some Post classes. Moreover, this polytope is combinatorial, which
means that the adjacency graph of its vertices is Hamiltonian. In the last part
of the paper, we have given a characterization of faces of these polyhedra which
contain games with a core reduced to a singleton.

Still some issues would need a deeper analysis, especially in adjacency
relations. While we have provided a characterization of adjacent vertices for
$\cBG_+(n)$, we did not perform this analysis for the extremal rays of
$\cBG(n)$, $\cBG_\alpha(n)$, nor for the facets of these polyhedra. This would
help to solve the projection problem we intend to address in a future work:
finding the closest balanced game for a given non-balanced game.

Another topic of future research would be to study the set of monotone balanced
games. Monotone games are games satisfying the following property: if
$S\subseteq T$, then $v(S)\leqslant v(T)$. As $v(\varnothing)=0$, these games
are nonnnegative and therefore form a subset of $\cBG_+(n)$.
This class of games has a great importance: they are known under the name of
{\it capacities} (Choquet \cite{cho53}) in decision theory and nonadditive
integral theory (see, e.g., \cite{sch89} and the monograph
\cite{gra16}). Unfortunately, the analysis of the set of balanced capacities
reveals to be extremely difficult: with $n=4$ the number of vertices is already
equal to 9002, and most of them are not 0-1-valued. Finding an analytical
characterization of them seems to be challenging.

\section*{Acknowledgements}

This work has been supported in part by the Spanish Grant PID2021-124933NB-I00.

\bibliographystyle{plain}

%\bibliography{../BIB/fuzzy,../BIB/grabisch,../BIB/general}

\appendix

\section{Proofs of Section~\ref{sec:bgalpha}}
\subsection{Proof of Theorem~\ref{th:5}}

\begin{enumerate}

\item Let us prove that $\cBG_{\alpha }(n)$ is an affine cone, i.e., the translation of a cone.
Observe that a particular game in $\cBG_\alpha(n)$ is $\alpha u _{\{n\}}$, i.e.,
the unanimity game centered on $\{n\}$ and multiplied by $\alpha$.

Define $\cC_\alpha(n):=\cBG_\alpha(n)-\alpha u _{\{n\}}$ and let us show that

$$ \cC_{\alpha }(n)=\{ v': \sum_{S\in\cB}\lambda^\cB_Sv'(S)\leq 0\} .$$

Pick $v\in \cBG_\alpha(n)$ and
consider $v':=v-\alpha u _{\{n\}}$. Then, for any $\cB\in\gB^*(n)$,
  \[
\sum_{S\in\cB}\lambda^\cB_Sv'(S) = \sum_{S\in\cB}\lambda^\cB_Sv(S) -
\sum_{S\in\cB,S\ni n}\lambda^\cB_S\alpha
  \leqslant \alpha -\alpha = 0.
  \]
 Hence, the claim is proved. Conversely, suppose that a game $v''$ satisfies

\begin{equation}\label{eq:6}
  \sum_{S\in\cB}\lambda^\cB_Sv''(S)\leqslant 0, \forall \cB\in\gB^*(n).
\end{equation}

Then, $v''+\alpha u _{\{n\}}$ satisfies
  \[
\sum_{S\in\cB}\lambda^\cB_S\left[ v''(S)+\alpha u _{\{n\}}(S)\right] =
\sum_{S\in\cB}\lambda^\cB_Sv''(S) +\alpha\leqslant \alpha,
  \]
for any $\cB\in\gB^*(n)$, i.e., $v''+\alpha u _{\{n\}}\in \cBG_\alpha(n)$ and thus $v''\in \cC_\alpha(n)$.

Finally, take $v\in \cC_\alpha(n)$ and
$\beta\geqslant 0$. We
have $\beta v\in \cC_\alpha(n)$ because it satisfies (\ref{eq:6}). Therefore, $\cC_\alpha(n)$ is a cone.

\item Proceeding as for the proof of Theorem~\ref{th:3}, the affine space contained
  in $\cBG_\alpha(n)$ is obtained by replacing $v(N)$ by $\alpha$ in the
  equations giving the lineality space of $\cBG(n)$. This yields the system

\begin{align*}
  \sum_{i\in N}v(\{i\}) &= \alpha\\
  v(S)+\sum_{i\in N\setminus S}v(\{i\}) & = \alpha, \quad S\subset N, |S|>1.
\end{align*}

To obtain the corresponding vector space, we just replace $\alpha$ by 0, thus obtaining:

\begin{align*}
  \sum_{i\in N}v(\{i\}) &= 0\\
  v(S)+\sum_{i\in N\setminus S}v(\{i\}) & = 0, \quad S\subset N, |S|>1.
\end{align*}

The set of solutions for this system is, expressed in terms of
$v(\{1\}),\ldots,v(\{n-1\})$,

\[
\{(v_1,\ldots,v_{n-1},-\sum_{i=1}^{n-1}v_i,v_1+v_2,\ldots,\underbrace{\sum_{i\in
  S}v_i}_{S\not\ni n},\ldots,\underbrace{-\sum_{i\not\in S}v_i}_{S\ni n},\ldots)\mymid v_1,\ldots,v_{n-1}\in\RR\}.
\]

This yields the basis $(w_i)_{i\in N\setminus\{n\}}$ given above.
\end{enumerate}

\subsection{Proof of Theorem~\ref{th:6}}
We follow the same steps as for Theorem \ref{th:4}.

Observe that since $r_n=-\delta_{\{n\}}$, it could be considered as a ray of type
$r_S$ with $S=\{n\}$. We follow this way in steps 2 and 3 of the proof.

\begin{enumerate}
\item The $2n-2$ extremal rays corresponding to the lineality space come from its
basis given in Th.~\ref{th:5}.

\item Consider $S\subset N$, $|S|>1$ or $S=\{n\}$, and let us show that $r_S$ is
extremal. Clearly
(\ref{eq:6}) is satisfied and $r_S(\{i\})=0$ for all $i\in N\setminus \{n\}$,
hence $r_S$ is a ray of $\cC^0_\alpha(n)$. Suppose it is not extremal. Then there
exist two rays $r,r'\in\cC_\alpha^0(n)$ such that $r_S=r+r'$. Suppose that
$r(T)>0$, say, $r(T)=1$ for some $T\neq S$, $1<|T|<n$ or $T=\{n\}$. Then, $r'(T)=-1$.

Using the partition $\{T,(N\setminus T)^\bot\}$, the corresponding inequality in
(\ref{eq:6}) for $r$ becomes either $1\leqslant 0$ if $n\in T$, or $1+r(\{n\} )\leqslant
0$ if $n\not\in T$. The first case being a contradiction, let us study the
second case. For this case, it follows  that $r(\{n\})\leqslant -1$, which in turn implies $r'(\{n\})\geqslant
1$. Considering then the partition $\{N^\bot\}$, the corresponding inequality in
(\ref{eq:6}) for $r'$ becomes $0+r'({\{n\}} )\leqslant 0$, a contradiction.

The case $r(T)<0$ can be treated similarly. We conclude that $r$ must have zero
coordinates, possibly excepted for $T=S$. Observe that as before, $r(S)>0$ is not possible
since the inequality for $\{S,(N\setminus S)^\bot\}$ would not be satisfied (proceed as above
with $T$). Hence, $r_S$ is extremal.

\item Let us take $i\in N\setminus\{n\}$ and show that $r_i$ is an extremal ray.
First, we prove it is a ray in $\cC_\alpha^0(n)$. By definition, $r_i(\{j\})=0$ for all
$j\neq n$. It remains to prove that $r_i$ is a solution of (\ref{eq:6}).
Taking any m.b.c. $\cB$, the inequality becomes:
\[
\sum_{\substack{S\in\cB\\S\ni i\\S\not\ni n\\|S|>1}}\lambda^\cB_S -
\sum_{\substack{S\in\cB\\S\not\ni i\\S\ni n}}\lambda^\cB_S =
\sum_{\substack{S\in\cB\\S\ni i\\S\not\ni n\\|S|>1}}\lambda^\cB_S - \Big(1-
\sum_{\substack{S\in\cB\\S\ni i\\S\ni n}}\lambda^\cB_S\Big)=\sum_{\substack{S\in\cB\\S\ni i\\|S|>1}}\lambda^\cB_S-1 \leqslant 0.
\]
This is obviously satisfied as $\sum_{\substack{S\in\cB\\S\ni
    i}}\lambda^\cB_S=1$ and $\lambda^\cB_S\geq 0.$ Observe that the inequality is tight iff $\cB\not\ni\{i\}$.

To show that $r_i$ is extremal, we need to show that the set of solutions of the
system of tight inequalities has dimension 1. We have already observed that the
inequality is tight iff $\{i\}\not\in\cB$. Let us consider
this system in $v,$ i.e.
\begin{equation}\label{eq:13}
\sum_{S\in \cB} \lambda_S^{\cB} v(S)= 0, \forall \cB \in \gB^*(n) \text{ s.t. } \{i\}\not\in\cB,
\end{equation}
with $v\in \cG_\alpha(n)$ s.t. $v(i)=0, i\ne n.$ Take the partition
$\cB=\{S,S^\bot )$ with $i\in S$. The corresponding equality in (\ref{eq:13})
reads $v(S)=0$ if $S\ni n$, and $v(S)+v(\{n\})=0$ otherwise, i.e.,
$v(S)=-v(\{n\})$. It follows that the set of solutions has the form
$\{(\underbrace{\beta}_{\{n\}},\underbrace{0\cdots 0}_{S\ni n,|S|>1},\underbrace{-\beta\cdots-\beta}_{S\not\ni
  n,|S|>1}),\beta\in\RR\}$, and hence has dimension 1.

\item It remains to prove that there is no other extremal ray. Consider a ray $w\in
\cC_\alpha^0(n)$, hence satisfying $w(\{i\})=0$ for all $i\in N\setminus\{n\}$
and (\ref{eq:6}). If $w$ is not a conic combination of $r_S$ and $r_i$,
$S\subset N$, $|S|>1$, $i\in N$, the system (\ref{eq:s1}) has no solution in $\alpha_S,\alpha_i$.
Using definitions of $r_S,r_i$ and omitting coordinates for singletons (except
$n$) and $N$ in $r_S,r_i,w$, we obtain

\begin{align}
  -\alpha_S -\sum_{i\not\in S}\alpha_i &= w(S), \quad S\subset N,|S|>1, S\ni n\nonumber\\
  -\alpha_S +\sum_{i\in S}\alpha_i &= w(S), \quad S\subset N,|S|>1, S\not \ni n\nonumber\\
  \alpha_S &\geqslant 0, \quad S\subset N, |S|>1\nonumber \\
  \alpha_i &\geqslant 0, \quad i\in N.\nonumber
\end{align}

We may denote the whole system by $A\alpha\geqslant b$ in matrix notation (with
some abuse).  Then, by Farkas' Lemma, this system has no solution iff there exists a vector $[y \ z\ t]$ with
coordinates $y_S\in\RR$, $S\subseteq N, |S|>1$, $z_T\geqslant 0$, $T\subset
N,|T|>1$, and $t_i\geqslant 0$, $i\in N$, such that $[y \ z\ t]^\top A=0$ and
$[y\ z\ t]^\top b>0$. Observe that the only vector $[y \ z\ t]$ solution of $[y
  \ z\ t]^\top A=0$ is:

\[
y_S=1\ (S\subset N, |S|>1), \quad z_S=1 \ (S\subset N,
|S|>1), \quad t_1=\cdots=t_n=0.
\]

For this solution, we obtain
\[
[y\ z\ t]^\top b = \sum_{S\subset N,|S|>1}w(S).
\]

Considering again the balanced collection $\{S\subset N,|S|>1\}$ with balancing weights $\frac{1}{2^{n-1}-2},$ it follows that
\[
\sum_{S\subset N,|S|>1}w(S)\leqslant 0.
\]

Therefore, $[y\ z\ t]^\top b\leqslant 0$ and the second system has no solution. Hence, the first system has always a
solution.
\end{enumerate}
\end{document}